\documentclass[10pt,twoside,a4paper]{article}

\usepackage{amsmath,amsfonts,euscript,amssymb,amsthm,a4,times}
\usepackage{esint}
\usepackage{mathrsfs}
\usepackage{enumerate}
\usepackage{appendix}
\usepackage[affil-it]{authblk}

\usepackage{pstricks}

\theoremstyle{plain} 
\newtheorem{theorem}{Theorem}

\theoremstyle{plain}

\theoremstyle{plain} 
\newtheorem{lemma}[theorem]{Lemma}

\theoremstyle{plain} 
\newtheorem{cor}[theorem]{Corollary}

\theoremstyle{plain}

\theoremstyle{definition}

\theoremstyle{definition} 
\newtheorem{definition}[theorem]{Definition}

\theoremstyle{plain}

\theoremstyle{definition}

\theoremstyle{definition}

\theoremstyle{definition} 
\newtheorem*{acknowledgements}{Acknowledgements}

\theoremstyle{plain}

\theoremstyle{definition} 

\newcommand{\R}{\mathbb{R}}
\newcommand{\Rn}{\mathbb{R}^{n}}

\newcommand{\RN}{\mathbb{R}^{N}}

\newcommand{\RNn}{\mathbb{R}^{N \times n}}

\newcommand{\scFl}{\mathscr{F}_{\textrm{loc}}}

\newcommand{\tostar}{\stackrel{\ast}{\rightharpoonup}}

\newcommand{\M}{\mathbb{R}^{N \times n}}
\newcommand{\ud}{\,\mathrm{d}}

\newcommand{\Ln}{\mathscr{L}^{n}}
\newcommand{\Leb}{\mathscr{L}}

\newcommand{\finf}{f^{\infty}}
\newcommand{\BV}{\textrm{BV}}

\newcommand{\Hd}{\mathscr{H}}

\allowdisplaybreaks

\numberwithin{equation}{section}
\numberwithin{theorem}{section}
\title{On growth conditions for quasiconvex integrands} 
\author{Parth Soneji\thanks{E-mail address: \texttt{soneji@math.lmu.de}}}
\affil{\small Ludwig Maximilians University Munich \\Theresienstr. 39, 80333 Munich, Germany}

\date{}

\providecommand{\bysame}{\leavevmode\hbox to3em{\hrulefill}\thinspace}
\providecommand{\MR}{\relax\ifhmode\unskip\space\fi MR }
\providecommand{\MRhref}[2]{%
  \href{http://www.ams.org/mathscinet-getitem?mr=#1}{#2}
}
\providecommand{\href}[2]{#2}

\begin{document}
\maketitle
\vspace*{-20pt}
\begin{abstract}
\noindent We prove that, for $1\leq p< 2$, if a $W^{1,p}$-quasiconvex integrand $f\colon\M\rightarrow\R$ has linear growth from above on the rank-one cone, then it must satisfy this growth for all matrices in $\M$. An immediate corollary of this is, for example, that there can be no quasiconvex integrand that has genuinely superlinear $p$ growth from above for $1<p<2$, but only linear growth in rank-one directions. This result was first conjectured in \cite{soneji14}, with some partial results given. 

The key element of this proof involves constructing a Sobolev function which maps points in a cube to some one-dimensional frame, and moreover preserves boundary values. This construction is an inductive process on the dimension $n$, and involves using a Whitney decomposition.

This technique also allows us to generalise this result for $W^{1,p}$-quasiconvex integrands where $1\leq p < k\leq \min\{n,N\}$. 


\noindent\textit{Key words:} Quasiconvexity, $W^{1,p}$-- quasiconvexity, growth conditions, determinant
% 
% 
% This technique in fact provides us with the more general result that if $f$ is a $W^{1,p}$-quasiconvex function for some $1\leq p < k$, where $2\leq k \leq \min\{n,N\}$, and satisfies the growth condition
% \begin{equation*}
%  0\leq f(\xi)\leq C(1+|\xi|^{q})
% \end{equation*}
% for some exponent $1\leq q < k$, for all matrices $\xi\in\R^{N\times n}$ such that $\textrm{rank}(\xi)\leq k-1$, then $f$ satisfies this growth condition for all matrices $\xi\in\R^{N\times n}$. 
\end{abstract}

\section{Introduction}

Consider the variational integral
\begin{equation*}
F(u,\Omega):=\int_{\Omega}f(\nabla u(x))\ud x\,\textrm{,}%\label{functionalFintro}
\end{equation*}
where $\Omega$ is a bounded, open subset of $\Rn$, $u\colon\Omega\rightarrow\RN$ is a vector-valued function, $\nabla u$ denotes the Jacobian matrix of $u$ and $f$ is a Borel measurable function defined on the space $\M$ of all real $N\times n$ matrices, with $N$, $n\geq 2$. 

The notion of quasiconvexity, introduced by Morrey in \cite{morreyqc}, is of central importance in the modern theory of the Calculus of Variations. Recall that $f\colon\RNn\rightarrow\R$ is said to be quasiconvex if it is locally bounded, and for some nonempty bounded, open set $\Omega\subset\Rn$ with $\Ln(\partial\Omega)=0$ we have
\begin{equation}
 \int_{\Omega}f(\xi+\nabla\phi(x))\ud x \geq \Ln(\Omega)f(\xi)\label{qcdef}
\end{equation}
for all $\xi\in\RNn$ and all test functions $\phi\in W^{1,\infty}_{0}(\Omega ;\RN)$. Moreover, it is well-known that if property \eqref{qcdef} holds for some suitable set $\Omega$ then it holds for all bounded open sets $D\subset\Rn$ with $\Ln(\partial D)=0$.

The classical lower semicontinuity result for quasiconvex integrands states that if the integrand $f\colon\RNn\rightarrow\R$ is quasiconvex and satisfies the growth condition
\begin{equation}
 0\leq f(\xi)\leq L(1+|\xi|^{p})\label{introgrowth}
\end{equation}
for all $\xi\in\RNn$, for some constant $L>0$, and some exponent $1\leq p <\infty$, then $F(\cdot,\Omega)$ is sequentially weakly lower semicontinuous in $W^{1,p}(\Omega;\RN)$. Note that here and throughout this paper we may take any norm we wish on the space $\M$. For example, we can set $|\xi|:= \big(\sum_{i=1}^{N}\sum_{j=1}^{n} \xi_{ij}^{2}\big)^{\frac{1}{2}}$. 

 This theorem is essentially due to Morrey \cite{morreyqc, morreybook}, who proved sequential weak* lower semicontinuity of $F$ in $W^{1,\infty}(\Omega;\RN)$ in the case where the quasiconvex integrand $f$ need only be locally bounded. Refinements were made most notably by Meyers \cite{Meyers}, Acerbi and Fusco \cite{acerbifusco}, and Marcellini \cite{Marcelliniapp}. In fact, it has been shown that lower semicontinuity obtains even if $f$ takes on negative values, provided it satisfies the lower bound $f(\xi)\geq -l(1+|\xi|^{q})$ for some fixed constant $l>0$ where (if $p>1$) $1\leq q <p$, or $q=p=1$.

\bigskip\noindent Now let us recall the notion of $W^{1,p}$-- quasiconvexity, introduced and studied in a well-known paper by Ball and Murat \cite{BallMurat}, which generalises in a natural way the quasiconvexity condition of Morrey.
\begin{definition}\label{w1pqc}
Let $f \colon \M \to \R\cup\{+\infty\}$ be Borel-mesurable and bounded below, and fix $1\leq p\leq\infty$ and $\xi  \in \M$. Then $f$ is said to be $W^{1,p}$ --quasiconvex at $\xi$ if and only if for some nonempty bounded, open set $\Omega\subset\Rn$ with $\Ln(\partial \Omega)=0$ we have
\begin{equation}
 \int_{\Omega}f(\xi+\nabla\phi(x))\ud x \geq \Ln(\Omega)f(\xi)\label{W1peq}
\end{equation}
for all $\phi\in W^{1,p}_{0}(\Omega;\RN)$. $f$ is said to be $W^{1,p}$-- quasiconvex if this inequality holds for all $\xi\in\M$. 
\end{definition}

From this definition, we can easily see that if $f$ is  $W^{1,p}$-- quasiconvex, then it is also  $W^{1,q}$-- quasiconvex for all $p\leq q\leq \infty$. Thus $W^{1,1}$-quasiconvexity is the strongest condition and $W^{1,\infty}$-quasiconvexity is the weakest. In their paper, Ball and Murat demonstrate, similarly to Morrey, that $W^{1,p}$ --quasiconvexity is a necessary condition for sequential weak lower semicontinuity in $W^{1,p}(\Omega;\RN)$ (weak* if $p=\infty$). Again, if property \eqref{W1peq} holds for some suitable set $\Omega$ then it holds for all bounded, open sets $D\subset\Rn$ with $\Ln(\partial D)=0$. 

\medskip\noindent Note that in their definition, Ball and Murat require slightly different pre-conditions on $f$: that it is Borel measurable and bounded below, and moreover is permitted to take the value $+\infty$. In light of these discrepancies, the definitions of $W^{1,\infty}$-- quasiconvexity and Morrey's classical definition of quasiconvexity given above (which is also the one given in, for example, \cite{dacorogna}) may not precisely coincide. Therefore, for ease of exposition, let us from now on assume that $f\colon\RNn\rightarrow\R$ is 
\begin{itemize}
 \item[-] continuous (hence locally bounded), and
\item[-] bounded below (i.e. $f\geq -l$ for some fixed constant $l>0$).
\end{itemize}
 In this case, these two notions are equivalent, and will henceforth just be called ``quasiconvexity''. Indeed, as is well known, (the classical definition of) quasiconvexity implies rank-one convexity, which implies separate convexity, which implies (local Lipschitz) continuity. Moreover, the focus of this paper is on growth conditions from above. We remark that by modifying these conditions, some slight generalisations and alternative statements of the results we give are possible: please refer to the end of this section for a brief discussion of these. 

Under these assumptions, the following well-known fact relates this property to Morrey's classical definition of quasiconvexity, and is straightforward to prove.

\begin{lemma}\label{W1pgrowthrel}
 Let $f\colon\M\to\R$ be a continuous function satisfying the growth condition
 \begin{equation}
  -l\leq f(\xi)\leq L(1+|\xi|^{p})\label{modgrowth}
 \end{equation}
for some exponent $1\leq p<\infty$, fixed constants $l$, $L>0$, and all $\xi\in\M$. Then $f$ is quasiconvex if and only if it is $W^{1,p}$-- quasiconvex. 
\end{lemma}
\begin{proof}
 Since $W^{1,\infty}_{0}(\Omega;\RN)\subset W^{1,p}_{0}(\Omega;\RN)$ for any bounded, open set $\Omega\subset\Rn$, clearly $W^{1,p}$ --quasiconvexity implies quasiconvexity. 

The other implication follows from the fact that if $f$ satisfies the given growth condition, then $F$ is strongly continuous in $W^{1,p}$. Let $\Omega\subset\Rn$ be open and bounded, with $\Ln(\partial \Omega)=0$, and let $\phi\in W^{1,p}_{0}(\Omega;\RN)$. There exists a sequence $(\phi_{j})\subset C^{\infty}_{c}(\Omega;\RN)$ such that $\phi_{j}\rightarrow \phi$ strongly in $W^{1,p}(\Omega;\RN)$. Hence there is a subsequence $(\phi_{j_{k}})$ such that $\nabla\phi_{j_{k}}(x)\rightarrow\phi(x)$ for $\Ln$-almost all $x\in \Omega$. 

Since $f$ is continuous, $f(\nabla(\phi_{j_{k}}(x))\rightarrow f(\nabla(x))$ almost everywhere too, and hence (since $\Ln(\Omega)<\infty$) in measure. By Vitali's convergence theorem, $(|\phi_{j_{k}}|^{p})$ is equi-integrable, and so by the growth condition \eqref{modgrowth}, so is $(f(\nabla\phi_{j_{k}}))$. Thus by Vitali $f(\nabla\phi_{j_{k}})$ converges strongly in $L^{1}(\Omega)$ to $f(\phi)$. This holds for any subsequence of $(\phi_{j})$, so in fact the full sequence $f(\nabla\psi_{j})$ converges to $f(\nabla\phi)$ in $L^{1}(\Omega)$. Hence we conclude, since $f$ is quasiconvex, that
 \begin{equation*}
  \int_{\Omega}f(\xi+\nabla\phi(x))\ud x = \lim_{j\rightarrow\infty}\int_{\Omega}f(\xi+\nabla\psi_{j}(x))\ud x \geq \Ln(\Omega)f(\xi)\,\textrm{.}
 \end{equation*}
\end{proof}

\bigskip\noindent We now state the main theorems proved in this paper.

\begin{theorem}\label{thm1} Let $1\leq p <2$. Suppose $f\colon\M\to\R$ is a $W^{1,p}$-- quasiconvex function that satisfies the linear growth condition
 \begin{equation}
  f(\xi)\leq L(1+|\xi|)\label{G1}
 \end{equation}
whenever rank$(\xi)\leq 1$. Then in fact $f$ satisfies \eqref{G1} for all matrices $\xi\in\M$ (for perhaps a larger constant $L$), and hence is $W^{1,1}$-- quasiconvex. 
\end{theorem}

This result was first proved in the simpler case $n=N=2$ in \cite{soneji14}: here, we are able to include all $n$, $N\geq 2$. Morover, the proof we provide allows us to further generalise the theorem as follows. 

\begin{theorem}\label{thm2}
 Suppose $f\colon\R^{N\times n}\rightarrow\R$ is a $W^{1,p}$-- quasiconvex function for some $1\leq p < k$, where $2\leq k \leq \min\{n,N\}$. Suppose also that $f$ satisfies the growth condition
\begin{equation}
f(\xi)\leq L(1+|\xi|^{q})\label{Gq}
\end{equation}
for some exponent $1\leq q < k$, for all matrices $\xi\in\R^{N\times n}$ such that $\textrm{rank}(\xi)\leq k-1$. Then in fact $f$ satisfies \eqref{Gq} for all matrices $\xi\in\R^{N\times n}$ (for perhaps a larger constant $L$), and hence is $W^{1,q}$-- quasiconvex. 
\end{theorem}

As a consequence of Lemma \ref{W1pgrowthrel}, the following results immediately follow. In this context, when we say that an integrand $f\colon\M\rightarrow\R$ has  \textit{genuinely growth of order} $p$ (from above), we mean that it satisfies \eqref{modgrowth} for such an exponent $1\leq p<\infty$, and moreover there exists $\xi_{0}\in\M$ such that
\begin{equation*}
 \limsup_{t\rightarrow\infty} \frac{f(t\xi_{0})}{1+|t\xi_{0}|^{p}} >0\,\textrm{,}
\end{equation*}
so, in particular, no exponent $q$,  $1\leq q<p$, would be large enough to bound $f$ in \eqref{modgrowth}.

\begin{cor}\label{cor1}
 There can be no quasiconvex function $f\colon\M\rightarrow\R$ that has genuinely superlinear growth of order $1<p<2$, but only linear growth from above - i.e. \eqref{G1} - on rank-one matrices.
\end{cor}

In this connection we refer also to \cite{sverak}, where it was shown that there do indeed exist quasiconvex function of subquadratic growth that are not polyconvex (and hence not convex); the example provided here is in fact isotropic (it has the same growth in all directions).  Corollary \ref{cor1} generalises to:

\begin{cor}\label{cor2}
 Let $2\leq k \leq \min\{n,N\}$, and $1\leq q < p < k$. There can be no quasiconvex function $f\colon\M\rightarrow\R$ that has genuinely $p$-growth, but only $q$-growth from above - i.e. \eqref{Gq} - on matrices $\xi\in\R^{N\times n}$ such that $\textrm{rank}(\xi)\leq k-1$.
\end{cor}

In light of the discussion below concerning the integrands involving the determinant, the following result may also be of interest.

\begin{cor}\label{nodet}
 Suppose $f\colon\R^{N\times n}\rightarrow\R$ is a $W^{1,p}$-- quasiconvex function for some $1\leq p < k$, where $2\leq k \leq \min\{n,N\}$. If $f$ satisfies the upper bound, for some $\gamma\geq 0$
\begin{equation*}
f(\xi)\leq \gamma
\end{equation*}
for all matrices $\xi\in\R^{N\times n}$ such that $\textrm{rank}(\xi)\leq k-1$, then $f$ satisfies this upper bound, with the same constant $\gamma$, on all matrices in $\M$. In particular, if $f$ is non-negative and satisfies $f(\xi)=0$ whenever $\textrm{rank}(\xi)\leq k-1$, then in fact $f$ is identically zero on all of $\M$. 
\end{cor}
% 
% \noindent  Recall that the \textit{quasiconvex envelope} of a continuous function $f\colon\M\rightarrow\R$ is defined as
% \begin{equation*}
%  (Qf)(\xi):= \sup\{g(\xi) : g\leq f\,\textrm{ and }g\,\textrm{ quasiconvex }\,\}\,\textrm{.}
% \end{equation*}
% If $f$ is bounded below by some quasiconvex function (so, for example, if $f$ is non-negative), this is also a quasiconvex function - see \cite{}. As a consequence of this corollary, for example, 

\medskip\noindent The key ingredient in the proof of Theorem \ref{thm1} is a map $w$ defined on the cube $Q=(-1,1)^{n}$, that maps this cube onto a one-dimensional frame and preserves boundary values. Moreover, this $w$ belongs to $W^{1,p}(Q;\Rn)$, for any $1\leq p <2$, and $\textrm{rank}(\nabla w(x))=1$ for almost all $x\in Q$. 

The construction of $w$ is an inductive process on the dimension $n$. For $n=2$ the construction is straightforward, and was also given in \cite{soneji14}: it is the mapping to the boundary of the square. To obtain $w$ for higher dimensions, we assume as an inductive hypothesis that a suitable map $w_{n-1}$ has been constructed for dimension $n-1$. We then construct a map $u$ which maps points in $Q$ to its boundary (one of $2n$ faces), and then apply the map $w_{n-1}$ to each $(n-1)$-dimensional face. We then take a Whitney decomposition of $Q$, and apply such a map $u$ (appropriately scaled, and slightly modified) on each cube of the decomposition, resulting in a map that is the identity on the boundary. This construction is described in detail, with all required properties proved, in the subsequent section. The proof of Theorem \ref{thm2} involves a straightforward generalisation of such a construction.

\subsection{Some background and motivation}

The determinant enables us to readily produce examples of $W^{1,p}$-- quasiconvex functions for integer exponent $p$: for example, if $n=N$, and
\begin{equation*}
 f(\xi)= |\det{\xi}|\,\textrm{,}
\end{equation*}
then $f$ is a quasiconvex (in fact polyconvex - see \cite{ballelastic}) function that satisfies \eqref{modgrowth} for $p=n$, and hence by Lemma \ref{W1pgrowthrel} it is $W^{1,n}$-- quasiconvex. More generally, if $1\leq k \leq\min\{n,N\}$, then by considering the determinant of some $k\times k$ minor we can also provide an example of a $W^{1,k}$-- quasiconvex function. It is interesting to note that such a function $f$ also has the property that
\begin{equation*}
 f(\xi)=0\quad\textrm{whenever rank}(\xi)\leq k-1\,\textrm{.}
\end{equation*}
For example, taking the case $k=2$, any $W^{1,2}$-- quasiconvex function given by the determinant of a $2\times 2$ minor will vanish on all matrices of rank one or below. This observation may lead us to ask whether a similar such property might hold for non-integer $p$. That is, for instance, if $1\leq p<2$ and $f$ is $W^{1,p}$-- quasiconvex, then can we expect any different behaviour, such as growth, on matrices of rank one compared to other matrices? The results contained in this paper establish that, contrary to the determinant, the growth conditions (from above) for such a function $f$ on rank one matrices, in fact in some sense determine growth contitions for all general matrices. 

In a similar vein, it is interesting to note that Corollary \ref{nodet} immediately implies that if $n=N$, $0<\alpha<1$, and $f(\xi)=|\det(\xi)|^{\alpha}$, then the quasiconvex envelope $Qf$ of $f$, defined as 
\begin{equation*}
(Qf)(\xi):= \sup\{g(\xi) : g\leq f\,\textrm{ and }g\,\textrm{ quasiconvex }\,\} \,\textrm{,}
\end{equation*}
is identically zero. More generally, if $f$ is the modulus of the determinant of a minor to the power $\alpha$, then $(Qf)\equiv 0$. This result is already well-known, and more usually proved using properties of the rank-one convex envelope.

\medskip\noindent Another motivation for such a property might be found in \cite{soneji14}: this paper considers the ``Lebesgue-Serrin Extension''
\begin{equation*}
 \mathscr{F}_{\textrm{loc}}(u,\Omega):= \inf_{(u_{j})}\bigg\{ \liminf_{j\rightarrow\infty}\int_{\Omega}f(\nabla u_{j}(x))\ud x\, \left| 
\!\!\begin{array}{rl}
& (u_{j})\subset W_{\textrm{loc}}^{1,p}(\Omega, \RN) \\
 & u_{j}\tostar u\,\,\textrm{in }\BV(\Omega, \RN) 
\end{array} \right. \bigg\} \,\textrm{,}
\end{equation*}
where $u$ is a function of Bounded Variation, for some exponent $1\leq p<\infty$. It was proved by Ambrosio and Dal Maso in \cite{ambrosiodalmaso}, and Fonseca and M{\"u}ller in \cite{fonsecamuller}, that if $f$ is quasiconvex and satisfies \eqref{introgrowth} for $p=1$, then the extension has the integral representation
\begin{equation}
 \scFl(u,\Omega) = \int_{\Omega} f(\nabla u (x))\ud x + \int_{\Omega}\finf \bigg(\frac{D^{s}u}{|D^{s}u|}(x)\bigg)\,\ud |D^{s}u|\,\textrm{,} 
\end{equation}
where $\nabla u$ is the density of the absolutely continuous part of the measure $Du$ with respect to Lebesgue measure, $D^{s}u$ is the singular part of $Du$, $\frac{D^{s}u}{|D^{s}u|}$ is the Radon-Nikod{\'y}m derivative of the measure $D^{s}u$ with respect to its total variation $|D^{s}u|$, and $f^{\infty}$ denotes the \emph{recession function of} $f$, defined as
\begin{equation*}
 f^{\infty}(\xi):= \limsup_{t\rightarrow\infty}\frac{f(t\xi )}{t}\,\textrm{.}
\end{equation*} 
This integral representation in the convex case was proved earlier by Goffman and Serrin in \cite{goffman}: in this setting, no growth assumptions on the integrand are required. 

Focusing on the quasiconvex case, there have been some more recent results obtained in the \textit{non-standard growth} setting: that is, $f$ satisfies \eqref{introgrowth} for some $p>1$, but we still consider semicontinuity properties with respect to weak* convergence in BV. Such problems where the space of the convergence is below the growth exponent of the integrand were considered in the Sobolev Space setting by, among many others, Bouchitt{\'e}, Fonseca and Mal{\'y} in \cite{relax2, relax1}. In \cite{JanBV}, Kristensen shows that when $f$ is quasiconvex and satisfies the growth condition \eqref{introgrowth} for $1\leq <\frac{n}{n-1}$, $\scFl$ satisfies the lower bound
\begin{equation}
 \scFl(u,\Omega)\geq \int_{\Omega}f(\nabla u(x))\ud x\,\textrm{,}\label{introJanlb}
\end{equation}
whenever $u\in\BV(\Omega;\RN)$. In \cite{soneji}, a lower semicontinuity result in the sequential weak* topology of BV is obtained for $1<p<2$. This result requires us to assume additionally that the maps $(u_{j})$ are bounded uniformly in $L^{q}_{\textrm{loc}}$ for $q$ suitably large, and that the limit map $u$ is sufficiently regular. 

However, neither of these results incorporate the singular part of the measure $Du$ for a map $u\in\BV(\Omega;\RN)$. One particular problem arising here is that if $f$ has superlinear growth in all directions (for example, if it is isotropic), then the recession function $f^{\infty}$, which is crucial for describing the behaviour of the Lebesgue-Serrin extension on the singular part, will just be infinity. Hence, one might wish to somehow ensure that $f$ satisfies
\begin{equation*}
 \finf \bigg(\frac{D^{s}u}{|D^{s}u|}(x)\bigg)<\infty\quad\textrm{for }\,|D^{s}u|\textrm{-a.a. }x\in\Omega\,\textrm{.}
\end{equation*}
Due to Alberti's famous rank-one theorem in \cite{albertirank1}, the term $\frac{D^{s}u}{|D^{s}u|}(x)$ is rank-one for $|D^{s}u|$-almost all $x\in\Omega$. Therefore, a natural additional assumption one might make is that the integrand $f$, whilst it may enjoy superlinear growth in general, should satisfy
\begin{equation*}
 \finf (\xi)<\infty\quad\textrm{whenever rank}(\xi)\leq 1\,\textrm{.}
\end{equation*}
This is equivalent to saying that $f$ should have at most linear growth in rank-one directions. In \cite{soneji14}, it was proved that if $f$ is continuous (not necessarily) quasiconvex, satisfies \eqref{introgrowth} for $1\leq p<2$, but has linear growth in rank-one directions, then, for general $u\in\BV(\Omega;\RN)$, the extension satisfies the upper bound
\begin{equation*}
 \scFl(u,\Omega)\leq L(\Ln(\Omega) + |Du|(\Omega))\,\textrm{.}
\end{equation*}
However this result, combined with the lower bound of Kristensen \eqref{introJanlb}, implies that if $f$ is additionally quasiconvex and $1\leq p<\frac{n}{n-1}$, then it cannot have genuinely superlinear growth in any direction. This can be seen by considering the linear map $u(x)=\xi x$ for any $\xi\in\M$, which yields, for any bounded open set $\Omega\subset\Rn$,
\begin{equation*}
 \Ln(\Omega) f(\xi)\leq \scFl(u,\Omega)\leq L\Ln(\Omega) (1+ |\xi|)\,\textrm{,}
\end{equation*}
 which shows that $f$ has linear growth in all directions. Consequently, linearity on the rank-one cone is in fact not a good assumption. This observation (which is in fact a weaker form of Corollary \ref{cor1}) provided the first motivation for this paper.

\subsection{Remarks on the conditions imposed on the integrand}

 In the statements of the results above, we have assumed that the integrand $f$ is continuous, real-valued  (so the values $\pm\infty$ are excluded), and bounded below. Here, we shall provide a short discussion of how these conditions may be modified to obtain slightly different conclusions, all of which may be easily obtained using the methods in this paper. 

\medskip\noindent We first note that the conclusions regarding growth in Theorem \ref{thm1} and Theorem \ref{thm2} remain unchanged if we just adopt the definition of $W^{1,p}$-quasiconvexity from \cite{BallMurat} as stated in Definition \ref{w1pqc}. However, in order to conclude that the integrands are $W^{1,1}$-- and $W^{1,q}$-- quasiconvex respectively, we make use of the fact that $f$ is upper semicontinuous. 

 This is due to a variant of Lemma \ref{W1pgrowthrel} that is proved in \cite{BallMurat}, which states that if $f$ is upper semicontinuous and satisfies the growth condition \eqref{modgrowth} for some exponent $1\leq p<\infty$ (and hence cannot be $+\infty$), then it is $W^{1,p}$-- quasiconvex if and only if it is $W^{1,\infty}$-- quasiconvex. The proof is similar the one above, and involves a straightforward application of Fatou's Lemma.

\medskip\noindent Moreover, in the subsequent corollaries stated above, we may remove the requirement that the integrand be bounded below. Note that the proof of Lemma \ref{W1pgrowthrel} also tells us that if $f$ is quasiconvex and satisfies, for some exponent $1\leq p<\infty$,
\begin{equation}
  |f(\xi)|\leq L(1+|\xi|^{p})\label{modreal}
\end{equation}
(so it is not necessarily bounded below), then it also satisfies the quasiconvexity inequality \eqref{qcdef} for all $\xi\in\RNn$ and all test functions $\phi\in W^{1,p}_{0}(\Omega ;\RN)$. Thus, in the statement of Corollary \ref{cor1}, for example, we may say that if a quasiconvex function satisfies \eqref{modreal} for some $1<p<2$, but only has linear growth from above on rank-one matrices, then it must have linear growth from above on all matrices (but the possibility of having superlinear, subquadratic growth from below is not ruled out). Similarly, a generalised type of statement can be formulated for Corollary \ref{cor2}.

%Quasiconvexity is commonly defined (see, for example \cite{dacorogna}) as a property of real-valued locally bounded integrands $f\colon\M\rightarrow\R$. However, as is well known, quasiconvexity implies rank-one convexity, which implies separate convexity, which implies (local Lipschitz) continuity. Hence no discrepancies arise here. 
% 
% In their famous paper \cite{BallMurat}, Ball and Murat impose slightly different pre-conditions on $f$. Namely, they consider integrands that are bounded below and Borel measurable, and moreover may take the value $+\infty$. In this case, we only need the integrand to be upper semicontinuous for Lemma \ref{W1pgrowthrel} to hold, since the growth condition \eqref{modgrowth} becomes
% \begin{equation*}
%    -l\leq f(\xi)\leq L(1+|\xi|^{p})\,\textrm{.}
% \end{equation*}
% 
% In light of such considerations, we remark that it is possible for other statements and corollaries to be formulated using the core technique presented here, in addition to the ones that we have stated, catering to different kinds of conditions on $f$. For example, in the statement of Theorem \ref{thm1}, we may instead suppose $f$ is upper semicontinuous and bounded below, and obtain the conclusion that $f$ must be $W^{1,1}$-quasiconvex. Similarly, in Theorem \ref{thm2}, we can 

\begin{acknowledgements}
 The author would like to gratefully acknowledge the support of the Ludwig-Maximilians-University, Munich.  He would also like to extend thanks to Jan Kristensen and Lars Diening for numerous helpful discussions.
\end{acknowledgements}

\section{Proofs of the main results}

As indicated above, the proof of Theorem \ref{thm1} depends on the following Lemma.

\begin{lemma}\label{lemma2}
Let $1\leq p <2$. Let $Q=Q^{(n)}=(-1,1)^{n}$. Then there exists a map $w\colon Q\rightarrow\Rn$ such that $w\in W^{1,p}(Q;\Rn)$, $w$ maps $Q$ onto (a one-dimensional frame within) $Q$, $w$ equals the identity map $\iota$ on $\partial Q$ in the sense of traces, and $\textrm{rank}(\nabla w(x))=1$ for $\Ln$-almost all $x\in Q$.
\end{lemma}

The majority of this section is devoted to proving this result, which we shall do in several steps. The proof of Theorem \ref{thm2} depends on the following, generalised version of this lemma.

\begin{lemma}\label{lemma3}
 Let $2\leq k\leq n$, and $1\leq p <k$. Let $Q=Q^{(n)}=(-1,1)^{n}$. Then there exists a map $\tilde{w}\colon Q\rightarrow\Rn$ such that $\tilde{w}\in W^{1,p}(Q;\Rn)$, $\tilde{w}$ maps $Q$ onto (a $(k-1)$-dimensional frame within) $Q$, $\tilde{w}$ equals the identity map $\iota$ on $\partial Q$, and $\textrm{rank}(\nabla \tilde{w}(x))\leq k-1$ for $\Ln$-almost all $x\in Q$.
\end{lemma}

Let us first show how the theorems follow from these lemmas. 
\begin{proof}[Proof of Theorem \ref{thm1}]
%  First focus on the case $N=n$. Let $\xi$ be a general matrix in $\R^{n\times n}$. Now define the map $w_{\xi}\colon D\rightarrow \Rn$ by
% \begin{equation*}
% w_{\xi}(x):= \xi w(x)\,\textrm{, }x\in D\,\textrm{,} 
% \end{equation*}
% where $w$ is the map in Lemma \ref{lemma2}. Then $w_{\xi}\in W^{1,p}(D;\Rn)$, and $w_{\xi}=\xi$ on $\partial D$. Therefore, since $f$ is $W^{1,p}$-quasiconvex, we have
% \begin{equation*}
%  \int_{D}f(\nabla w_{\xi})\ud x \geq |D|f(\xi)\,\textrm{.}
% \end{equation*}
% Moreover, $\textrm{rank}(\nabla w_{\xi}(x))\leq \textrm{rank}\nabla w(x)=1$ for almost all $x\in D$. Hence by the assumption in the Theorem, we have
% \begin{align*}
%  \int_{D}f(\nabla w_{\xi})\ud x &\leq C\int_{D} 1+ |\nabla w_{\xi}|\ud x\\
% &\leq C\Bigg(|D| + |\xi|\int_{D}|\nabla w|\ud x\Bigg)\,\textrm{.}
% \end{align*}
% Since the $L^{1}$-norm of $w$ is a finite constant, we combine the estimates to get
% \begin{equation*}
%  f(\xi)\leq C(1+|\xi|)
% \end{equation*}
% for some constant $C>0$. 

Let $\xi$ be a general matrix in $\M$. Now define the map $w_{\xi}\colon Q\rightarrow \RN$ by
\begin{equation*}
w_{\xi}(x):= \xi w(x)\,\textrm{, }x\in Q\,\textrm{,} 
\end{equation*}
where $w\colon Q\rightarrow\Rn$ is the map in Lemma \ref{lemma2}. Then we have 
\begin{equation*}
 \nabla w_{\xi}(x) = \xi \nabla w(x)\,\textrm{,}
\end{equation*}
so certainly $w_{\xi}\in W^{1,p}(Q;\RN)$. Moreover, for $x\in \partial Q$, we have (where $I_{n}$ denotes the $n\times n$ identity matrix) 
\begin{equation*}
\nabla w_{\xi}(x)= \xi I_{n}x = \xi x\,\textrm{.} 
\end{equation*}
Therefore, since $f$ is $W^{1,p}$-- quasiconvex (and $\Ln(Q)=1$),
\begin{equation}
 \int_{Q}f(\nabla w_{\xi})\ud x \geq f(\xi)\,\textrm{.}\label{lb}
\end{equation}
In addition, $\textrm{rank}(\nabla w_{\xi}(x))\leq \textrm{rank}\nabla w(x)=1$ for almost all $x\in Q$. Hence by the assumption in the theorem, we have
\begin{align}
 \int_{Q}f(\nabla w_{\xi})\ud x &\leq L\int_{Q} 1+ |\nabla w_{\xi}|\ud x\nonumber \\
&\leq L\Bigg(1 + |\xi|\int_{Q}|\nabla w|\ud x\Bigg)\,\textrm{.}\label{ub}
\end{align}
Since the $L^{1}$-norm of $\nabla w$ is a finite constant, we combine the estimates \eqref{lb} and \eqref{ub} to get
\begin{equation*}
 f(\xi)\leq L'(1+|\xi|)
\end{equation*}
for some constant $L'>0$. 

%If $n > N\geq 2$, then the proof follows by embedding the constructed function $w$ into a higher dimensional space.

% then we can consider the space $\R^{n\times n}$ of $n\times n$ matrices as a linear subspace of $\M$ by
% \begin{equation*}
%   \left(\begin{array}{cccc} a_{1,1} & a_{1,2} & \cdots & a_{1,n} \\
%                          a_{2,1} & a_{2,2} & \cdots & a_{2,n}  \\
% \vdots & \vdots & \vdots & \vdots \\
%                          a_{n,1} & a_{n,2} & \cdots & a_{n,n} \end{array}\right)\longmapsto   \left(\begin{array}{cccc} a_{1,1} & a_{1,2} & \cdots & a_{1,n} \\
%                          a_{2,1} & a_{2,2} & \cdots & a_{2,n}  \\
% \vdots & \vdots & \vdots & \vdots \\
%                          a_{n,1} & a_{n,2} & \cdots & a_{n,n}\\
% 0 & 0 & \cdots & 0\\
% \vdots & \vdots & \ddots & \vdots \\
% 0 & 0 & \cdots & 0  \end{array}\right)
%  \end{equation*}
% Let $P\colon\M\rightarrow\R^{n\times n}$ denote orthogonal projection (so, for $\xi\in\M$, $P\xi$ is the $n\times n$ matrix obtained by eliminating the bottom $N-n$ rows of $\xi$). Now, given a general matrix $\xi\in\M$, define the map $\hat{w}_{\xi}\colon D\rightarrow\RN$ by 
% \begin{equation*}
% \hat{w}_{\xi}(x):= \xi w(x)\,\textrm{, }x\in D\,\textrm{,} 
% \end{equation*}
% where $w$ is the map in Lemma \ref{lemma2}.
\end{proof}

\noindent The proofs of Theorem \ref{thm2} and Corollary \ref{nodet} follow in exactly the same way, by using the map $\tilde{w}$ from Lemma \ref{lemma3} instead of $w$ from Lemma \ref{lemma2}.

\medskip\noindent For clarity of exposition, we shall focus on proving Lemma \ref{lemma2}, since this provides us with more concrete assumptions and parameters. We shall then indicate how the proof of Lemma \ref{lemma3} just involves a straightforward generalisation of the technique contained here.

\subsection{Proof of the main lemma}
As stated earlier, the construction of $w$ is an inductive process on the dimension $n$. The base case $n=2$ is straightforward. To construct $w$ ($=w_{n}$) for higher $n$, we first construct a map $u\in W^{1,p}(Q^{(n)};\Rn)$, such that $u$ first maps $Q^{(n)}$ onto $\partial Q^{(n)}$, and then we ``apply $w_{n-1}$'' to each of the $2n$ faces of $\partial Q^{(n)}$. Then we have, using an induction hypothesis, that $\textrm{rank}(\nabla u(x))=1$ for $\Ln$-almost all $x\in Q^{(n)}$. This map $u$ will form a key ``building-block'' for $w$ as stated in the Lemma. We define, for $x=(x_{1},\ldots x_{m})\in\R^{m}$, 
\begin{equation*}
 \|x\|:= \max\{|x_{1}|,\ldots |x_{m}|\}\,\textrm{,}
\end{equation*}
(so $\|\cdot\|$ is just the infinity norm in $\R^{m}$). We shall allow the $m$ to vary during this proof, but in an unambiguous way. Here is a brief outline of the steps used:
\begin{itemize}
 \item[\textbf{Step 1:}] Here we construct $w=w_{2}$ in the base case $n=2$. This was also shown in \cite{soneji14}, and is simply the mapping of all points in the square (apart from $(0,0)$) to its boundary. 
\item[\textbf{Step 2:}] We suppose that for $2\leq d <n$ we have constructed a map $w_{d}$ on $Q^{(d)}:=(-1,1)^{d}$ that satisfies the properties of Lemma \ref{lemma2} for dimension $d$. We use this hypothesis to construct the building block map $u$.  This first maps points in $(-1,1)^{n}$ onto the $(n-1)$-dimensional ``faces'', followed by the map $w_{n-1}$ applied to each face. We observe that for $n=3$, $u$ just maps the cube $Q^{(3)}$ to its one-dimensional edges. 
\item[\textbf{Step 3:}] The map $u$ as it has been constructed is still not exactly what we will need in the construction of $w_{n}$. Namely, on at least on one of the $2n$ faces of $Q^{(n)}$, we do not just apply $w_{n-1}$ but in fact split this face into $2^{n-1}$ ``subfaces'', and apply (an appropriately scaled) $w_{n-1}$ individually to each of these. Here we describe this modification, which will be called $v$. 
\item[\textbf{Step 4:}] We now describe a standard Whitney decomposition of the larger cube $D=(-3,3)^{n}$, containing the smaller cube $(-1,1)^{n}$.
\item[\textbf{Step 5:}] We use this Whitney decomposition to construct the map $w_{n}$ on $D$ (and hence, by appropriate scaling, on $Q$). It essentially involves a translation, dilation, and rotation of the (modified) map $u$ applied to individual cubes within the Whitney decomposition described in Step 4. We show that the map $w_{n}$ satisfies all the required properties.
\item[\textbf{Step 6:}] We indicate how this construction may be generalised to establish Lemma \ref{lemma3}. 
\end{itemize}

\bigskip\noindent\textbf{Step 1: Construction of $w$ in the base case $n=2$}

\nopagebreak\smallskip\noindent For $n=2$ the construction of $w$ is straightforward. Write $Q^{(2)}=(-1,1)^{2}$ and define the map $w_{2}\colon Q^{(2)}\setminus\{(0,0)\}\rightarrow\R^{2}$ as
\begin{equation*}
 w_{2}(x_{1},x_{2}):= \frac{(x_{1},x_{2})}{\|(x_{1},x_{2})\|}\,\textrm{.}
\end{equation*}
Then it is well known that $w_{2}\in W^{1,p}(Q^{(2)};\R^{2})$ for any $1\leq p<2$ (for example, see \cite{BallMurat}). Since $w_{2}$ maps $Q^{(2)}\setminus\{(0,0)\}$ into $\partial Q^{(2)}$, we can conclude that $\det \nabla w_{2} (x)= 0$ on $Q^{(2)}\setminus\{(0,0)\}$, so rank$(\nabla w_{2} (x))\leq 1$ for $\Leb^{2}$-almost all $x\in Q^{(2)}$. Indeed, on $Q^{(2)}\setminus\{|x_{1}|=|x_{2}|\}$, $w_{2}$ has the strong derivatives
\begin{equation*}
 \frac{\partial w_{2}^{j}}{\partial x_{i}} = \frac{\delta{j,i}}{\|x\|}- \frac{x_{i}}{\|x\|^{2}}\frac{\partial \|x\|}{\partial x_{i}}\,\textrm{,}
\end{equation*}
where
\begin{equation*}
 \frac{\partial \|x\|}{\partial x_{i}} =\left\{
\begin{array}{cl}
\textrm{sign }x_{i} &\textrm{if }|x_{i}|=\|x\|\, \textrm{,}\\
0 & \textrm{if }|x_{i}|\neq\|x\|\,\textrm{.}
\end{array} \right.
\end{equation*}
It is easy to establish, using the Gauss-Green Theorem, that this is in fact a weak derivative of $w_{2}$ on all of $Q^{(2)}$, and that it is $p$-integrable. Moreover, for any $x=(x_{1},x_{2})\in Q^{(2)}$ with $|x_{1}|\neq |x_{2}|$, the $i$th column, where $\|(x_{1},x_{2})\|= |x_{i}|$, has zero entries. Hence $w_{2}$ satisfies the required properties for Lemma \ref{lemma2} in the case $n=2$.

\bigskip\noindent \textbf{Step 2: Inductive step and construction of $u$ for higher dimensions} 

\nopagebreak\smallskip\noindent We now deal with higher dimensions inductively. Suppose Lemma \ref{lemma2} holds for every dimension $2\leq d < n$. That is, for every such $d$ there exists a map $w_{d}\colon Q^{(d)}\rightarrow\R^{d}$ such that $w\in W^{1,p}(Q^{(d)};\R^{d})$ for $1\leq p <2$, $w_{d}$ equals the identity map $\iota$ on $\partial Q^{(d)}$, and $\textrm{rank}(\nabla w_{d}(x))=1$ for $\Ln$-almost all $x\in Q^{(d)}$. We now construct the building block map $u$ on $Q^{(n)}$ as follows. Note that $Q^{(n)}$ has $2n$ ``$(n-1)$-faces'', $F_{1},\ldots F_{2n}$, say, where each $F_{k}$ is a set of the form
\begin{equation*}
\{(x_{1},\ldots x_{i_{n}})\in\R^{d} : x_{i_{1}}= \pm 1\,\textrm{, } \|(x_{i_{2}},\ldots x_{i_{n-1}})\|\leq 1\}
\end{equation*}
where $\{i_{1},\ldots , i_{n}\} = \{1,\ldots , n\}$. So every such face is isometrically isomorphic to $\overline{Q^{(n-1)}}$. We define the map $u$ to be first the map
\begin{equation}
 (x_{1},\ldots,x_{n})\mapsto \frac{(x_{1},\ldots,x_{n})}{\|(x_{1},\ldots, x_{n})\|}\,\textrm{,}\label{facemapn}
\end{equation}
which maps points in $Q^{(n)}\setminus\{0\}$ to one of the faces $F_{k}$, followed by the map $w_{n-1}$ ``applied to the face''. That is, 
\begin{equation*}
 u(x):= \Phi_{k}^{-1}\big(w_{n-1}(\Phi_{k}(x/\|x\|))\big)
\end{equation*}
where $x/\|x\| \in F_{k}$ and $\Phi_{k}:F_{k}\rightarrow \overline{Q^{(n-1)}}$ is the isomorphism identifying that face with (the closure of) $Q^{(n-1)}$. So if
\begin{equation}
 F_{1}= \{(x_{1},\ldots , x_{n})\in\R^{n} : x_{1}=1\,\textrm{, } \|(x_{2},\ldots , x_{n})\|\leq 1\}\,\textrm{,}\label{facegenn}
\end{equation}
then
\begin{equation*}
 \Phi_{1}(1,x_{2},\ldots , x_{n}) = (x_{2},\ldots , x_{n})\in \overline{Q^{(n-1)}}\,\textrm{.}
\end{equation*}

We shall now show that $u \in W^{1,p}(Q^{(n)};\R^{n})$. First note that $\overline{Q^{(n)}}$ may be expressed as the union of the closure of $2n$ cones, each cone corresponding to the face mapped-to by the expression in \eqref{facemapn}. We first consider the cone 
\begin{equation}
 C_{1}:= \{  (x_{1}\ldots,x_{n})\in Q^{(n)} : x_{1}> 0\,\textrm{, } \|(x_{2},\ldots x_{n})\|< x_{1}\}\,\textrm{,}\label{cone1n}
\end{equation}
and show $u\in W^{1,p}(C_{1};\R^{n})$. Elements in this cone first get mapped to the face $F_{1}$ from \eqref{facemapn}, and then $w_{n-1}$ is ``applied to the face''. Hence, for $x\in C_{1}\setminus \{(x_{2},\ldots x_{n})=(0,\ldots, 0)\}$ we have, using $\|x\| = x_{1} =:r$, 
\begin{equation}
 u(x) = \bigg(1, w_{n-1}\bigg(\frac{(x_{2},\ldots x_{n})}{r}\bigg)\bigg)\label{un}
\end{equation}
We verify, using the induction hypothesis that $w_{n-1}$ is weakly differentiable, that $u$ is also weakly differentiable on $C_{1}$ with weak derivative
\begin{equation}
 \nabla u (x) = \left(\begin{array}{ccc} 0 & 0 & 0 \\
                        0  & \multicolumn{2}{c}{ r^{-1}\nabla w_{n-1}\big(\frac{(x_{2},\ldots x_{d})}{r}\big)} \\
		      0  \end{array}\right)\,\textrm{.}\label{Dun}
\end{equation}
Moreover, using the inductive hypothesis that $w_{n-1}\in W^{1,p}(Q^{(n-1)};\R^{n-1})$,
\begin{align*}
 \int_{C_{1}}|\nabla u|^{p}\ud x &= \int_{0}^{1}\int_{\{\|(x_{2},\ldots x_{n})\|<r\}}|\nabla u|^{p}\ud\Hd^{n-1}\ud r \\
&= \int_{0}^{1}\int_{\{\|(x_{2},\ldots x_{n})\|<r\}}|\nabla(1, w_{n-1}((x_{2},\ldots x_{d})/r)|^{p}\ud\Hd^{n-1}(x_{2},\ldots x_{n})\ud r \\
&= \int_{0}^{1}r^{-p}\int_{\{\|y\|<r\}}|\nabla w_{n-1}(y/r)|^{p}\ud\Hd^{n-1}(y)\ud r \\
&= \int_{0}^{1}r^{n-1-p}\int_{Q^{(n-1)}}|\nabla w_{n-1}(y)|^{p}\ud\Hd^{n-1}(y)\ud r
 <\infty \,\textrm{.}
\end{align*}
so $u\in W^{1,p}(C_{1};\R^{n})$. Again using our inductive hypothesis and \eqref{Dun}, we establish that for $\Leb^{n}$-almost all $x\in C_{1}$, rank$(\nabla u(x)) \leq 1 $ (in fact, even $\mathscr{H}^{n-1}$-almost everywhere). By arguing similarly on all cones, we obtain $u\in W^{1,p}(C_{k};\R^{n})$ for $k=1,\ldots 2n$, and rank$\nabla u = 1$ almost everywhere. Now suppose $x\in\partial C_{k}$ for some cone $C_{k}$. Then it is also on the boundary of some other cone, and so there exist $i, j\in\{1,\ldots, n\}$ with $i\neq j$  such that $|x_{j}|=|x_{i}|=\|x\|$. Suppose without loss of generality that one cone is $C_{1}$ from \eqref{cone1n} and we have $x_{1}=x_{2}=\|x\|$, so the other cone is
\begin{equation}
 C_{2}:= \{  (x_{1}\ldots,x_{n})\in Q^{(n)} : x_{2}> 0\,\textrm{, } \|(x_{1}, x_{3}, \ldots x_{n})\|< x_{2}\}\,\textrm{.}\label{cone2n}
\end{equation}
Then, using the inductive hypothesis that $w_{n-1}$ is the identity on $\partial Q^{(n-1})$ (in the sense of traces), we see that if we consider $x$ as being on the boundary of $C_{1}$, we have, using \eqref{un},
\begin{equation*}
 u(x) = \bigg(1, \frac{x_{2}}{x_{1}},\frac{x_{3}}{x_{1}}\ldots \frac{x_{n}}{x_{1}}\bigg) = \bigg(\frac{x_{1}}{x_{2}}, 1,\frac{x_{3}}{x_{2}}\ldots \frac{x_{n}}{x_{2}}\bigg)\,\textrm{,}
\end{equation*}
the right hand side being the value of $u(x)$ if $x$ were considered to be on the boundary of $C_{2}$. Hence there are no discontinuities on the boundaries of the cones. Therefore, dealing with the zero point in the centre in the standard way using Gauss-Green, we can conclude that $u$ is weakly differentiable on all of $Q^{(n)}$ with derivative given by the expression in \eqref{Dun} (up to a permutation of coordinates), so $u\in W^{1,p}(Q^{(n)};\R^{n})$ with $\textrm{rank}(\nabla u(x))\leq 1$ for $\Leb^{n}$-almost all $x\in Q^{(n)}$. 

\bigskip\noindent \textbf{Remark on the case $n=3$} 

\smallskip\noindent For illustrative purposes, let us consider in particular the case $d=3$. Moreover, since in this case we are only using the simpler map $w_{2}$ from Step 1, things are more straightforward: the resulting map $u$ maps points in the cube to the one-dimensional edges. Note that the cube $Q^{(3)}=(-1,1)^{3}$ has $6$ ($2$-dimensional) faces, $F_{1},\ldots F_{6}$, say, where each face $F_{k}$ is isometrically isomorphic to $\overline{Q^{(2)}}$. We define the map $u$ as in the general case above, first applying the map $x/\|x\|$ for $x\in Q^{(3)}$, and then $w_{2}$ (i.e. $x/\|x\|$ in dimension $2$) to each face. So if we consider the cone
\begin{equation*}
 C_{1}:= \{  (x_{1},x_{2},x_{3})\in Q^{(3)} : x_{1}> 0\,\textrm{, } \|(x_{2},x_{3})\|< x_{1}\}\,\textrm{,}
\end{equation*}
then for $x\in C_{1}\setminus \{(x_{2},x_{3})=(0,0)\}$ we have
\begin{equation*}
 u(x) = \bigg(1, w_{2}\bigg(\frac{(x_{2},x_{3})}{x_{1}}\bigg)\bigg)= \bigg(1, \frac{x_{2}}{\|(x_{2},x_{3})\|},  \frac{x_{2}}{\|(x_{2},x_{3})\|}\bigg)
\end{equation*}
Just as in the general case, it is straightforward to verify that $u$ is weakly differentiable on $C_{1}$ with weak derivative given by the expression in \eqref{Dun}. In fact, this is even a strong derivative for $x\notin \{(x_{2},x_{3})=(0,0)\}$. Then, using the fact that $w_{2}\in W^{1,p}(Q^{(2)};\R^{2})$ in Step 1, and applying the same argument as in the general inductive step, we can show that $u\in W^{1,p}(C_{1};\R^{3})$. Again, by \eqref{Dun} and Step 1, we can see that for $\Leb^{3}$-almost all $x\in C_{1}$, rank$(\nabla u(x)) \leq 1$ (in fact, this holds even $\mathscr{H}^{2}$-almost everywhere).  By arguing in the same way on all other cones, we obtain $u\in W^{1,p}(C_{k};\R^{3})$ for each $k=1,\ldots 6$, and rank$\nabla u = 1$ $\Leb^{3}$-almost everywhere. Since $u$ is continuous on  
\begin{equation}
 (-1,1)^{3}\setminus \{ (x_{1},x_{2}, x_{3}) : x_{i}=x_{j}=0\,\textrm{ for some }i\neq j\}\,\textrm{,}\label{cont3}
\end{equation}
(i.e. off the axes of $\R^{3}$), there are no discontinuities along the boundaries of these cones. We deal with the singularity at the centre of the cube with Gauss-Green, and conclude that $u$ weakly differentiable on all of $Q^{(3)}$. Hence we can conclude that $u_{3}\in W^{1,p}(Q^{(3)};\R^{3})$ with $\textrm{rank}(\nabla u(x))\leq 1$ for $\Leb^{3}$-almost all $x\in Q^{(3)}$. 

Moreover, for each point $(x_{1},x_{2}, x_{3})$ in the set described in \eqref{cont3}, if we take an index $\{i_{1}, i_{2}, i_{3}\}=\{1,2 , 3\}$ such that $|x_{i_{1}}|\leq |x_{i_{2}}|\leq |x_{i_{3}}|$, then for $j=1,2,3$, (writing $u=(u^{(1)},u^{(2)}, u^{(3)})$)
\begin{equation*}
 u^{(i_{j})}(x) = \textrm{sign}(x_{i_{j}})\,\textrm{,}
\end{equation*}
and 
\begin{equation*}
 u^{(i_{1})}(x) = \frac{x_{i_{1}}}{|x_{i_{2}}|}
\end{equation*}
(i.e. all but one coordinate of $u(x)$ is $\pm 1$, as $u$ maps $x$ onto a one-dimensional edge of $(-1,1)^{3}$). So if, for instance, $x_{1}>x_{2}>x_{3} >0$, then the map u acts on $x$ as follows:
\begin{align*}
 &\Big(x_{1},x_{2}, x_{3}\Big) \mapsto \Big(1,\frac{x_{2}}{x_{1}},\frac{x_{3}}{x_{1}}\Big) \mapsto \Big(1,1, \frac{x_{3}}{x_{2}}\Big) \,\textrm{.}
\end{align*}
In light of this particular case, one might think that we could apply such a simpler construction for higher dimensions as well. That is, our building block map $u$ could just be the map that maps points in the general cube $Q^{(n)}$ to its one-dimensional frame. So we apply the map ``$x/\|x\|$'' to points in the $n$-cube, and then again (in one dimension less) to the $(n-1)$-faces, then again to each of the $(n-2)$-faces of these, and so on ($n-1$ times in total). So, for example, if $x_{1}>x_{2}>\ldots > x_{n} >0$, then such a map would act on $x$ as follows:
\begin{align*}
 &\Big(x_{1},x_{2}, \ldots ,x_{n-1}, x_{n}\Big) \mapsto \Big(1,\frac{x_{2}}{x_{1}},\ldots , \frac{x_{n-1}}{x_{1}}, \frac{x_{n}}{x_{1}}\Big) \mapsto \Big(1,1,\ldots ,\frac{x_{n-1}}{x_{2}}, \frac{x_{n}}{x_{2}}\Big) \mapsto \ldots \mapsto  \Big(1,1,\ldots , 1, \frac{x_{n}}{x_{n-1}}\Big)\,\textrm{.}
\end{align*}
If this were our $u$, we can also verify that $u\in W^{1,p}(Q^{(n)};\Rn)$ and $\textrm{rank}\nabla u (x)\leq 1$ for $\Ln$-almost all $x\in Q^{(n)}$. However, the reason we do not use this map in general is because we cannot successfully apply the next step in our construction of $w$: see the remark following the step.

\bigskip\noindent \textbf{Step 3: Modification of $u$ in some cones of $Q^{(n)}$} 

\nopagebreak\smallskip\noindent  In our eventual construction of $w_{n}$, our ``building-block'' map will not in fact just apply $w_{n-1}$ to each face of an $n$-cube. In (at least) one cone, we will want to split its corresponding $(n-1)$-face into $2^{n-1}$ ``subfaces'', and apply $w_{n-1}$ on each subface instead. Consider $Q^{(n-1)}=(-1,1)^{n-1}$ and note that it can be split up into $2^{n-1}$ quadrants. That is,
\begin{equation*}
 \overline{Q^{(n-1)}}= \bigcup_{l=1}^{2^{n-1}}\overline{P_{l}}
\end{equation*}
where each $P_{l}$ is of the form
\begin{equation*}
 P_{l}= \{ (x_{1},\ldots , x_{n-1})\in Q^{(n-1)} : x_{i_{1}},\ldots , x_{i_{j}} >0\,\textrm{, }x_{i_{j+!}},\ldots , x_{i_{n-1}} <0\}
\end{equation*}
where $j\in \{0, 1, \ldots, n-1\}$ and $\{i_{1},\ldots i_{n-1}\} = \{1,\ldots , n-1\}$. So in fact 
\begin{equation*}
 P_{l}= z_{l}+ \tfrac{1}{2}Q^{(n-1)}
\end{equation*}
where $z_{l} = (\pm \frac{1}{2}, \pm \frac{1}{2},\ldots , \pm \frac{1}{2})\in\R^{n-1}$. Now define $v_{n-1}$ on $Q^{(n-1)}$ by letting
\begin{equation}
 v_{n-1}(x):= z_{l}+\tfrac{1}{2}w_{n-1}(2(x-z_{l}))\,\textrm{, }x\in P_{l}\,\textrm{,}\label{mapvnminus}
\end{equation}
for each of the $2^{n-1}$ cubes $P_{l}$. Note by the inductive hypothesis on $w_{n-1}$, $v_{n-1}(x)=x$ (in the sense of traces) whenever $x$ belongs to the boundary $\partial P_{l}$ for any subface $P_{l}$. Hence $v_{n-1}$ has no discontinuities on each of the boundaries $\partial P_{l}$, and so is weakly differentiable on $Q^{(n-1)}$, with
\begin{align*}
 \int_{Q^{(n-1)}}|\nabla v_{n-1}|^{p}\ud\mathscr{L}^{n-1} &= \sum_{l=1}^{2^{n-1}}\int_{P_{l}}|\nabla v_{n-1}|^{p}\ud\mathscr{L}^{n-1}\\
&= \sum_{l=1}^{2^{n-1}}2^{1-n}\int_{Q^{(n-1)}}|\nabla w_{n-1}|^{p}\ud\mathscr{L}^{n-1}\\
&= \int_{Q^{(n-1)}}|\nabla w_{n-1}|^{p}\ud\mathscr{L}^{n-1}\,\textrm{.}
\end{align*}
We now modify the map $u$ as follows. Let $F_{k}$ be an $(n-1)$-face of $Q^{(n)}$, and let $C_{k}$ be the cone corresponding to points in $Q^{(n)}$ that are mapped to $F_{k}$ by $x\mapsto x/\|x\|$. Suppose we wish to modify $u$ in this cone. Then we split the face $F_{k}$ into $2^{n-1}$ subfaces $P_{1},\ldots P_{2^{n-1}}$ and define, for $x\in C_{k}$,
\begin{equation}
  v(x):= \Phi_{k}^{-1}\big(v_{n-1}(\Phi_{k}(x/\|x\|))\big)\label{mapv}
\end{equation}
where, as in Step 2, $\Phi_{k}$ is the isomorphism identifying that face with $Q^{(n-1)}$. We now argue in entirely the same way as above to show that $v\in W^{1,p}(C_{k};\Rn)$, and $\textrm{rank}\nabla v(x) = 1$ for almost all $x\in C_{k}$. In particular, note that
\begin{equation}
 v(x) = \frac{x}{\|x\|} = u(x)\quad\textrm{ whenever }\Phi_{k}\big(x/\|x\|\big)\in \partial P_{l}\label{bdPl}
\end{equation}
for any $1\leq l\leq 2^{n-1}$. Hence there are no discontinuities along the boundaries of the $2^{n-1}$ ``sub-cones'' $(C^{l}_{k})$ of $C_{k}$, where
\begin{equation*}
 C^{l}_{k}:= \{x\in C_{k} : x/\|x\| \in P_{l}\}\,\textrm{.}
\end{equation*}
Clearly, by considering a rotation of the domain, it does not matter what specific cone of $Q^{(n)}$ we apply this modification to. In the remaining cones of $Q^{(n)}$, we may either not modify $u$ (so $v=u$ off $C_{k}$), or we modify it just as described. In fact, as will be described in Step 5, in our eventual construction of $w_{n}$ we either modify $u$ in this way on only \emph{one} cone, or on \emph{all} cones. 

We now verify that $v\in W^{1,p}(Q^{(n)};\Rn)$. As in Step 2, it suffices to establish that there are no discontinuities on the boundaries of the $2n$ cones that comprise $Q^{(n)}$. Suppose without loss of generality that the cone $C_{1}$ as described in \eqref{cone1n} is the cone where we have modified $u$ as above, with corresponding face $F_{1}$ as in \eqref{facemapn}. Suppose $x\in\partial C_{1}$. Then also $x\in\partial P_{l}$ for some $l$, so by \eqref{bdPl} $v(x)=u(x)$, where $u$ is the unmodified map from Step 2. Since from that step we also know that $u$ has no discontinuities on $\partial C_{1}$, neither does $v$. 

The following diagram roughly illustrates how the map $u$ is modified on the cone $C_{1}$ after the map $x/\|x\|$ has been applied (but of course in $2$ dimensions we do not use this construction).  

\smallskip
\begin{center}
	\begin{pspicture}(-2.9,-2.9)(2.9,2.9)
		\psset{unit=0.9cm}
\psset{linewidth=1.0pt}
		\psframe(-3,-3)(3,3)
\psline(-3,-3)(3,3) \psline(3,-3)(-3,3) \psline[linestyle = dashed, linewidth=1pt ](0,0)(3,0) \psdot(3,0) \psdot(3,1.5)\rput(2.6,1.5){$z_{1}$} \psdot(3,-1.5)\rput(2.6,-1.5){$z_{2}$}\rput(0,1.5){$C_{2}$}\rput(0,-1.5){$C_{4}$}\rput(-1.5,0){$C_{3}$}\rput(1.5,0.7){$C^{1}_{1}$}\rput(1.5,-0.7){$C^{2}_{1}$}\psline[linestyle = dashed, linewidth=1pt ]{|-|}(3.3,-3)(3.3,-0.05)\psline[linestyle = dashed, linewidth=1pt ]{|-|}(3.3,0.05)(3.3,3)\rput[l](3.7,1.5){$\cong P_{1}$ - apply $\frac{1}{2}w_{n-1}(x-z_{1})$}\rput[l](3.7,-1.5){$\cong P_{2}$  - apply $\frac{1}{2}w_{n-1}(x-z_{2})$} \rput(0,3.2){$F_{2}$ - apply $w_{n-1}$}\rput(0,-3.3){$F_{4}$ - apply $w_{n-1}$} \rput[r](-3.2,0){$F_{3}$ - apply $w_{n-1}$}
\end{pspicture}
\end{center}
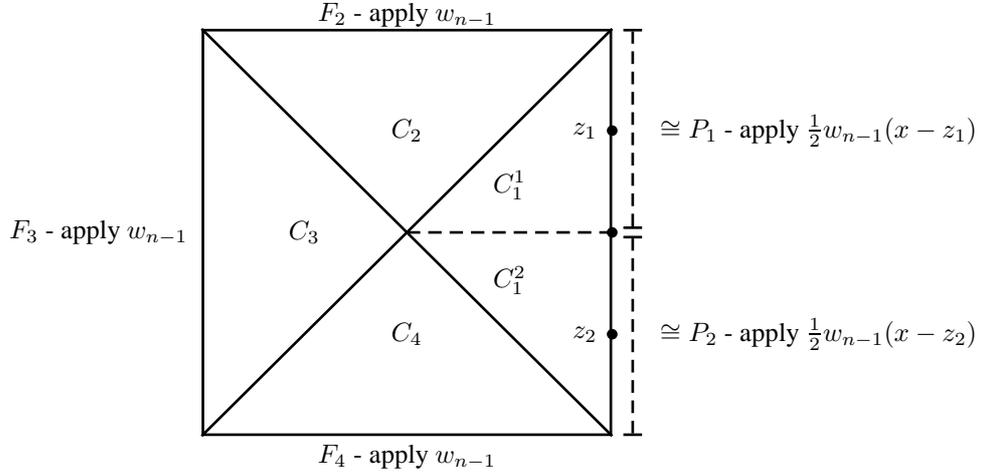
\nopagebreak\centerline{{\small \textbf{Figure 1:} How the map $u$ is modified on a cone of $Q^{(n)}$} }

\bigskip\noindent \textbf{Remark on possible simplification of $u$}

\smallskip\noindent As was noted in the remark to Step 2, one might intitially wish to construct $u$ as simply the mapping of elements in the cube $Q^{(n)}$ to its one-dimensional edges. However, for $n\geq 4$, the modification of $u$ described in Step 3 fails. 

For example, consider the point $x=(\frac{1}{2},\frac{1}{2},\frac{1}{4},0)\in Q^{(4)}$. Then $\|x\|=x_{1}=x_{2}$, so $x$ belongs to the boundary of the cones $C_{1}$ and $C_{2}$ in \eqref{cone1n} and \eqref{cone2n} respectively. Note that $x/\|x\| = (1,1,\frac{1}{2}, 0)$; if we were to use this simple map $u$, we note $(1,\frac{1}{2},0)$ is already on the boundary of $Q^{(3)}$, so we finally map $(\frac{1}{2},0)$  to $\partial Q^{(2)}$ to obtain $u(x)= (1,1,1,0)$. Suppose we wish to modify $u$ on $C_{1}$ only. In this case we split the cube $Q^{(3)}$ into $8$ subcubes with side length $1$ and centres $(\pm \frac{1}{2},\pm \frac{1}{2},\pm \frac{1}{2})$, and map onto the edges of this finer frame instead. However, then $(1,\frac{1}{2},0)$ already lies on this frame, so we have $v(x)= (1,1,\frac{1}{2}, 0) \neq u(x)$. This demonstrates that there is a discontinuity in $v$ between the boundary of the cones $C_{1}$ and $C_{2}$. 

\bigskip\noindent \textbf{Step 4: Whitney Decomposition of $D$} 

\smallskip\noindent Equipped with this map $v$, we are now almost in a position to define the map $w_{n}$ on $D=(-3,3)^{n}$. We take the standard Whitney Decomposition of $D$ into dyadic cubes whose side length is proportional to the distance from the boundary. We first start with the cube $Q_{1}=Q^{(n)}=(-1,1)^{n}$. Now consider the larger cube $Q_{2}:=(-2,2)^{n}$ and note that $\overline{Q_{2}}\setminus Q$ can be written as the union of the closure of $4^{n}-2^{n}$ cubes with side-length $1$. Each ($n-1$)-face of $Q_{1}$ will have $2^{n-1}$ smaller cubes adjacent to it. Each smaller cube will share one face with (the subset of) a face of the larger cube $Q_{1}$, and all other faces (but one) will be shared with a cube of the same size. Call the set of these smaller cubes $\mathscr{Q}_{2}$

Now let $Q_{3}:= (-\frac{5}{2},\frac{5}{2})^{n}$. Note $\overline{Q_{3}}\setminus Q_{2}$ can be written as the union of (the closure of) $2^{n}(5^{n}-4^{n})$ cubes of side-length $\frac{1}{2}$. Each cube of side length $1$ in the previous step will have its remaining exposed face touching the face of $2^{n-1}$ of these smaller cubes. Call the set of these cubes $\mathscr{Q}_{3}$

We continue inductively in this way. For each integer $k\geq 2$, let 
\begin{equation*}
Q_{k}:= \big((-3+ 2^{2-k}), (3-2^{2-k})\big)^{n} \,\textrm{.}
\end{equation*}
Then $\overline{Q_{k}}\setminus Q_{k-1}$ can be written as the union of (the closure of) $2^{n(k-2)}\Ln(Q_{k}\setminus Q_{k-1})$ cubes of side length $2^{2-k}$. Call this set of cubes $\mathscr{Q}_{k}$. Note that
\begin{equation*}
 D= \bigcup_{k=1}^{\infty}Q_{k} =\bigcup_{k=1}^{\infty} \bigcup_{Q\in\mathscr{Q}_{k}}\overline{Q}\,\textrm{.}  
\end{equation*}

For illustrative purposes, we provide a diagram of this decomposition below in the case where $n=2$, $D=(-3,3)^{2}$. However, recall that in dimension two we do not actually need to use this construction. 

\smallskip\begin{center}
	\begin{pspicture}(-3.7,-3.7)(3.7,3.7)
		\psset{unit=1.3cm}
\psset{linewidth=0.5pt}
		\psframe(-1,-1)(1,1)\psframe(-2,-2)(-1,-1)\psframe(-1,-2)(0,-1)\psframe(0,-2)(1,-1)\psframe(1,-2)(2,-1)\psframe(-2,1)(-1,2)\psframe(-1,1)(0,2)\psframe(0,1)(1,2)\psframe(1,1)(2,2)\psframe(-2,-1)(-1,0)\psframe(-2,0)(-1,1)\psframe(1,-1)(2,0)\psframe(1,0)(2,1)\psframe(-2.5,-2.5)(-2,-2)\psframe(-2,-2.5)(-1.5,-2)\psframe(-1.5,-2.5)(-1,-2)\psframe(-1,-2.5)(-0.5,-2)\psframe(-0.5,-2.5)(0,-2)\psframe(0,-2.5)(0.5,-2)\psframe(0.5,-2.5)(1,-2)\psframe(1,-2.5)(1.5,-2)\psframe(1.5,-2.5)(2,-2)\psframe(2,-2.5)(2.5,-2)\psframe(-2.5,2)(-2,2.5)\psframe(-2,2)(-1.5,2.5)\psframe(-1.5,2)(-1,2.5)\psframe(-1,2)(-0.5,2.5)\psframe(-0.5,2)(0,2.5)\psframe(0,2)(0.5,2.5)\psframe(0.5,2)(1,2.5)\psframe(1,2)(1.5,2.5)\psframe(1.5,2)(2,2.5)\psframe(2,2)(2.5,2.5)\psframe(-2.5,-2)(-2,-1.5)\psframe(-2.5,-1.5)(-2,-1)\psframe(-2.5,-1)(-2,-0.5)\psframe(-2.5,-0.5)(-2,0)\psframe(-2.5,0)(-2,0.5)\psframe(-2.5,0.5)(-2,1)\psframe(-2.5,1)(-2,1.5)\psframe(-2.5,1.5)(-2,2)\psframe(2,-2)(2.5,-1.5)\psframe(2,-1.5)(2.5,-1)\psframe(2,-1)(2.5,-0.5)
		\psframe(2,-0.5)(2.5,0)
		\psframe(2,0)(2.5,0.5)\psframe(2,0.5)(2.5,1)\psframe(2,1)(2.5,1.5)\psframe(2,1.5)(2.5,2)
\psframe(-2.75,-2.75)(-2.5,-2.5)\psframe(-2.5,-2.75)(-2.25,-2.5)\psframe(-2.25,-2.75)(-2,-2.5)\psframe(-2,-2.75)(-1.75,-2.5)\psframe(-1.75,-2.75)(-1.5,-2.5)\psframe(-1.5,-2.75)(-1.25,-2.5)\psframe(-1.25,-2.75)(-1,-2.5)\psframe(-1,-2.75)(-0.75,-2.5)\psframe(-0.75,-2.75)(-0.5,-2.5)\psframe(-0.5,-2.75)(-0.25,-2.5)\psframe(-0.25,-2.75)(0,-2.5)\psframe(0,-2.75)(0.25,-2.5)\psframe(0.25,-2.75)(0.5,-2.5)\psframe(0.5,-2.75)(0.75,-2.5)\psframe(0.75,-2.75)(1,-2.5)\psframe(1,-2.75)(1.25,-2.5)\psframe(1.25,-2.75)(1.5,-2.5)\psframe(1.5,-2.75)(1.75,-2.5)\psframe(1.75,-2.75)(2,-2.5)\psframe(2,-2.75)(2.25,-2.5)\psframe(2.25,-2.75)(2.5,-2.5)\psframe(2.5,-2.75)(2.75,-2.5)\psframe(-2.75,2.5)(-2.5,2.75)\psframe(-2.5,2.5)(-2.25,2.75)\psframe(-2.25,2.5)(-2,2.75)\psframe(-2,2.5)(-1.75,2.75)\psframe(-1.75,2.5)(-1.5,2.75)\psframe(-1.5,2.5)(-1.25,2.75)\psframe(-1.25,2.5)(-1,2.75)\psframe(-1,2.5)(-0.75,2.75)\psframe(-0.75,2.5)(-0.5,2.75)\psframe(-0.5,2.5)(-0.25,2.75)\psframe(-0.25,2.5)(0,2.75)\psframe(0,2.5)(0.25,2.75)
\psframe(0.25,2.5)(0.5,2.75)\psframe(0.5,2.5)(0.75,2.75)\psframe(0.75,2.5)(1,2.75)\psframe(1,2.5)(1.25,2.75)\psframe(1.25,2.5)(1.5,2.75)\psframe(1.5,2.5)(1.75,2.75)\psframe(1.75,2.5)(2,2.75)\psframe(2,2.5)(2.25,2.75)\psframe(2.25,2.5)(2.5,2.75)\psframe(2.5,2.5)(2.75,2.75)\psframe(-2.75,-2.75)(-2.5,-2.5)\psframe(-2.75,-2.5)(-2.5,-2.25)\psframe(-2.75,-2.25)(-2.5,-2)\psframe(-2.75,-2)(-2.5,-1.75)\psframe(-2.75,-1.75)(-2.5,-1.5)\psframe(-2.75,-1.5)(-2.5,-1.25)\psframe(-2.75,-1.25)(-2.5,-1)\psframe(-2.75,-1)(-2.5,-0.75)\psframe(-2.75,-0.75)(-2.5,-0.5)\psframe(-2.75,-0.5)(-2.5,-0.25)\psframe(-2.75,-0.25)(-2.5,0)\psframe(-2.75,0)(-2.5,0.25)\psframe(-2.75,0.25)(-2.5,0.5)\psframe(-2.75,0.5)(-2.5,0.75)\psframe(-2.75,0.75)(-2.5,1)\psframe(-2.75,1)(-2.5,1.25)\psframe(-2.75,1.25)(-2.5,1.5)\psframe(-2.75,1.5)(-2.5,1.75)\psframe(-2.75,1.75)(-2.5,2)\psframe(-2.75,2)(-2.5,2.25)\psframe(-2.75,2.25)(-2.5,2.5)\psframe(-2.75,2.5)(-2.5,2.75)\psframe(2.5,-2.75)(2.75,-2.5)\psframe(2.5,-2.5)(2.75,-2.25)\psframe(2.5,-2.25)(2.75,-2)
\psframe(2.5,-2)(2.75,-1.75)\psframe(2.5,-1.75)(2.75,-1.5)\psframe(2.5,-1.5)(2.75,-1.25)\psframe(2.5,-1.25)(2.75,-1)\psframe(2.5,-1)(2.75,-0.75)\psframe(2.5,-0.75)(2.75,-0.5)\psframe(2.5,-0.5)(2.75,-0.25)\psframe(2.5,-0.25)(2.75,0)\psframe(2.5,0)(2.75,0.25)\psframe(2.5,0.25)(2.75,0.5)\psframe(2.5,0.5)(2.75,0.75)\psframe(2.5,0.75)(2.75,1)\psframe(2.5,1)(2.75,1.25)\psframe(2.5,1.25)(2.75,1.5)\psframe(2.5,1.5)(2.75,1.75)\psframe(2.5,1.75)(2.75,2)\psframe(2.5,2)(2.75,2.25)\psframe(2.5,2.25)(2.75,2.5)\psframe(2.5,2.5)(2.75,2.75)
\psframe(-2.875,-2.875)(-2.75,-2.75)\psframe(-2.75,-2.875)(-2.625,-2.75)\psframe(-2.625,-2.875)(-2.5,-2.75)\psframe(-2.5,-2.875)(-2.375,-2.75)\psframe(-2.375,-2.875)(-2.25,-2.75)\psframe(-2.25,-2.875)(-2.125,-2.75)\psframe(-2.125,-2.875)(-2,-2.75)\psframe(-2,-2.875)(-1.875,-2.75)\psframe(-1.875,-2.875)(-1.75,-2.75)\psframe(-1.75,-2.875)(-1.625,-2.75)\psframe(-1.625,-2.875)(-1.5,-2.75)\psframe(-1.5,-2.875)(-1.375,-2.75)\psframe(-1.375,-2.875)(-1.25,-2.75)\psframe(-1.25,-2.875)(-1.125,-2.75)\psframe(-1.125,-2.875)(-1,-2.75)\psframe(-1,-2.875)(-0.875,-2.75)\psframe(-0.875,-2.875)(-0.75,-2.75)\psframe(-0.75,-2.875)(-0.625,-2.75)\psframe(-0.625,-2.875)(-0.5,-2.75)\psframe(-0.5,-2.875)(-0.375,-2.75)\psframe(-0.375,-2.875)(-0.25,-2.75)\psframe(-0.25,-2.875)(-0.125,-2.75)\psframe(-0.125,-2.875)(0,-2.75)\psframe(0,-2.875)(0.125,-2.75)\psframe(0.125,-2.875)(0.25,-2.75)\psframe(0.25,-2.875)(0.375,-2.75)\psframe(0.375,-2.875)(0.5,-2.75)\psframe(0.5,-2.875)(0.625,-2.75)\psframe(0.625,-2.875)(0.75,-2.75)
\psframe(0.75,-2.875)(0.875,-2.75)\psframe(0.875,-2.875)(1,-2.75)\psframe(1,-2.875)(1.125,-2.75)\psframe(1.125,-2.875)(1.25,-2.75)\psframe(1.25,-2.875)(1.375,-2.75)\psframe(1.375,-2.875)(1.5,-2.75)\psframe(1.5,-2.875)(1.625,-2.75)\psframe(1.625,-2.875)(1.75,-2.75)\psframe(1.75,-2.875)(1.875,-2.75)\psframe(1.875,-2.875)(2,-2.75)\psframe(2,-2.875)(2.125,-2.75)\psframe(2.125,-2.875)(2.25,-2.75)\psframe(2.25,-2.875)(2.375,-2.75)\psframe(2.375,-2.875)(2.5,-2.75)\psframe(2.5,-2.875)(2.625,-2.75)\psframe(2.625,-2.875)(2.75,-2.75)\psframe(2.75,-2.875)(2.875,-2.75)\psframe(-2.875,2.75)(-2.75,2.875)\psframe(-2.75,2.75)(-2.625,2.875)\psframe(-2.625,2.75)(-2.5,2.875)\psframe(-2.5,2.75)(-2.375,2.875)\psframe(-2.375,2.75)(-2.25,2.875)\psframe(-2.25,2.75)(-2.125,2.875)\psframe(-2.125,2.75)(-2,2.875)\psframe(-2,2.75)(-1.875,2.875)\psframe(-1.875,2.75)(-1.75,2.875)\psframe(-1.75,2.75)(-1.625,2.875)\psframe(-1.625,2.75)(-1.5,2.875)\psframe(-1.5,2.75)(-1.375,2.875)\psframe(-1.375,2.75)(-1.25,2.875)
\psframe(-1.25,2.75)(-1.125,2.875)
\psframe(-1.125,2.75)(-1,2.875)
\psframe(-1,2.75)(-0.875,2.875)\psframe(-0.875,2.75)(-0.75,2.875)\psframe(-0.75,2.75)(-0.625,2.875)\psframe(-0.625,2.75)(-0.5,2.875)\psframe(-0.5,2.75)(-0.375,2.875)\psframe(-0.375,2.75)(-0.25,2.875)\psframe(-0.25,2.75)(-0.125,2.875)\psframe(-0.125,2.75)(0,2.875)\psframe(0,2.75)(0.125,2.875)\psframe(0.125,2.75)(0.25,2.875)\psframe(0.25,2.75)(0.375,2.875)\psframe(0.375,2.75)(0.5,2.875)\psframe(0.5,2.75)(0.625,2.875)\psframe(0.625,2.75)(0.75,2.875)\psframe(0.75,2.75)(0.875,2.875)\psframe(0.875,2.75)(1,2.875)\psframe(1,2.75)(1.125,2.875)\psframe(1.125,2.75)(1.25,2.875)\psframe(1.25,2.75)(1.375,2.875)\psframe(1.375,2.75)(1.5,2.875)\psframe(1.5,2.75)(1.625,2.875)\psframe(1.625,2.75)(1.75,2.875)\psframe(1.75,2.75)(1.875,2.875)\psframe(1.875,2.75)(2,2.875)\psframe(2,2.75)(2.125,2.875)\psframe(2.125,2.75)(2.25,2.875)\psframe(2.25,2.75)(2.375,2.875)\psframe(2.375,2.75)(2.5,2.875)\psframe(2.5,2.75)(2.625,2.875)\psframe(2.625,2.75)(2.75,2.875)\psframe(2.75,2.75)(2.875,2.875)
\psframe(-2.875,-2.875)(-2.75,-2.75)\psframe(-2.875,-2.75)(-2.75,-2.625)\psframe(-2.875,-2.625)(-2.75,-2.5)\psframe(-2.875,-2.5)(-2.75,-2.375)\psframe(-2.875,-2.375)(-2.75,-2.25)\psframe(-2.875,-2.25)(-2.75,-2.125)\psframe(-2.875,-2.125)(-2.75,-2)\psframe(-2.875,-2)(-2.75,-1.875)\psframe(-2.875,-1.875)(-2.75,-1.75)\psframe(-2.875,-1.75)(-2.75,-1.625)\psframe(-2.875,-1.625)(-2.75,-1.5)\psframe(-2.875,-1.5)(-2.75,-1.375)\psframe(-2.875,-1.375)(-2.75,-1.25)\psframe(-2.875,-1.25)(-2.75,-1.125)\psframe(-2.875,-1.125)(-2.75,-1)\psframe(-2.875,-1)(-2.75,-0.875)\psframe(-2.875,-0.875)(-2.75,-0.75)\psframe(-2.875,-0.75)(-2.75,-0.625)\psframe(-2.875,-0.625)(-2.75,-0.5)\psframe(-2.875,-0.5)(-2.75,-0.375)\psframe(-2.875,-0.375)(-2.75,-0.25)\psframe(-2.875,-0.25)(-2.75,-0.125)\psframe(-2.875,-0.125)(-2.75,0)\psframe(-2.875,0)(-2.75,0.125)\psframe(-2.875,0.125)(-2.75,0.25)\psframe(-2.875,0.25)(-2.75,0.375)\psframe(-2.875,0.375)(-2.75,0.5)\psframe(-2.875,0.5)(-2.75,0.625)\psframe(-2.875,0.625)(-2.75,0.75)
\psframe(-2.875,0.75)(-2.75,0.875)
\psframe(-2.875,0.875)(-2.75,1)\psframe(-2.875,1)(-2.75,1.125)\psframe(-2.875,1.125)(-2.75,1.25)\psframe(-2.875,1.25)(-2.75,1.375)\psframe(-2.875,1.375)(-2.75,1.5)\psframe(-2.875,1.5)(-2.75,1.625)\psframe(-2.875,1.625)(-2.75,1.75)\psframe(-2.875,1.75)(-2.75,1.875)\psframe(-2.875,1.875)(-2.75,2)\psframe(-2.875,2)(-2.75,2.125)\psframe(-2.875,2.125)(-2.75,2.25)\psframe(-2.875,2.25)(-2.75,2.375)\psframe(-2.875,2.375)(-2.75,2.5)\psframe(-2.875,2.5)(-2.75,2.625)\psframe(-2.875,2.625)(-2.75,2.75)\psframe(-2.875,2.75)(-2.75,2.875)\psframe(2.75,-2.875)(2.875,-2.75)\psframe(2.75,-2.75)(2.875,-2.625)\psframe(2.75,-2.625)(2.875,-2.5)\psframe(2.75,-2.5)(2.875,-2.375)\psframe(2.75,-2.375)(2.875,-2.25)\psframe(2.75,-2.25)(2.875,-2.125)\psframe(2.75,-2.125)(2.875,-2)\psframe(2.75,-2)(2.875,-1.875)\psframe(2.75,-1.875)(2.875,-1.75)\psframe(2.75,-1.75)(2.875,-1.625)\psframe(2.75,-1.625)(2.875,-1.5)\psframe(2.75,-1.5)(2.875,-1.375)\psframe(2.75,-1.375)(2.875,-1.25)
\psframe(2.75,-1.25)(2.875,-1.125)\psframe(2.75,-1.125)(2.875,-1)
\psframe(2.75,-1)(2.875,-0.875)\psframe(2.75,-0.875)(2.875,-0.75)\psframe(2.75,-0.75)(2.875,-0.625)\psframe(2.75,-0.625)(2.875,-0.5)\psframe(2.75,-0.5)(2.875,-0.375)\psframe(2.75,-0.375)(2.875,-0.25)\psframe(2.75,-0.25)(2.875,-0.125)\psframe(2.75,-0.125)(2.875,0)\psframe(2.75,0)(2.875,0.125)\psframe(2.75,0.125)(2.875,0.25)\psframe(2.75,0.25)(2.875,0.375)\psframe(2.75,0.375)(2.875,0.5)\psframe(2.75,0.5)(2.875,0.625)\psframe(2.75,0.625)(2.875,0.75)\psframe(2.75,0.75)(2.875,0.875)\psframe(2.75,0.875)(2.875,1)\psframe(2.75,1)(2.875,1.125)\psframe(2.75,1.125)(2.875,1.25)\psframe(2.75,1.25)(2.875,1.375)\psframe(2.75,1.375)(2.875,1.5)\psframe(2.75,1.5)(2.875,1.625)\psframe(2.75,1.625)(2.875,1.75)\psframe(2.75,1.75)(2.875,1.875)\psframe(2.75,1.875)(2.875,2)\psframe(2.75,2)(2.875,2.125)\psframe(2.75,2.125)(2.875,2.25)\psframe(2.75,2.25)(2.875,2.375)\psframe(2.75,2.375)(2.875,2.5)\psframe(2.75,2.5)(2.875,2.625)\psframe(2.75,2.625)(2.875,2.75)\psframe(2.75,2.75)(2.875,2.875)
\psframe(-2.9375,-2.9375)(-2.875,-2.875)\psframe(-2.875,-2.9375)(-2.8125,-2.875)\psframe(-2.8125,-2.9375)(-2.75,-2.875)\psframe(-2.75,-2.9375)(-2.6875,-2.875)\psframe(-2.6875,-2.9375)(-2.625,-2.875)\psframe(-2.625,-2.9375)(-2.5625,-2.875)\psframe(-2.5625,-2.9375)(-2.5,-2.875)\psframe(-2.5,-2.9375)(-2.4375,-2.875)\psframe(-2.4375,-2.9375)(-2.375,-2.875)\psframe(-2.375,-2.9375)(-2.3125,-2.875)\psframe(-2.3125,-2.9375)(-2.25,-2.875)\psframe(-2.25,-2.9375)(-2.1875,-2.875)\psframe(-2.1875,-2.9375)(-2.125,-2.875)\psframe(-2.125,-2.9375)(-2.0625,-2.875)\psframe(-2.0625,-2.9375)(-2,-2.875)\psframe(-2,-2.9375)(-1.9375,-2.875)\psframe(-1.9375,-2.9375)(-1.875,-2.875)\psframe(-1.875,-2.9375)(-1.8125,-2.875)\psframe(-1.8125,-2.9375)(-1.75,-2.875)\psframe(-1.75,-2.9375)(-1.6875,-2.875)\psframe(-1.6875,-2.9375)(-1.625,-2.875)\psframe(-1.625,-2.9375)(-1.5625,-2.875)\psframe(-1.5625,-2.9375)(-1.5,-2.875)\psframe(-1.5,-2.9375)(-1.4375,-2.875)\psframe(-1.4375,-2.9375)(-1.375,-2.875)\psframe(-1.375,-2.9375)(-1.3125,-2.875)
\psframe(-1.3125,-2.9375)(-1.25,-2.875)\psframe(-1.25,-2.9375)(-1.1875,-2.875)\psframe(-1.1875,-2.9375)(-1.125,-2.875)\psframe(-1.125,-2.9375)(-1.0625,-2.875)\psframe(-1.0625,-2.9375)(-1,-2.875)\psframe(-1,-2.9375)(-0.9375,-2.875)\psframe(-0.9375,-2.9375)(-0.875,-2.875)\psframe(-0.875,-2.9375)(-0.8125,-2.875)\psframe(-0.8125,-2.9375)(-0.75,-2.875)\psframe(-0.75,-2.9375)(-0.6875,-2.875)\psframe(-0.6875,-2.9375)(-0.625,-2.875)\psframe(-0.625,-2.9375)(-0.5625,-2.875)\psframe(-0.5625,-2.9375)(-0.5,-2.875)\psframe(-0.5,-2.9375)(-0.4375,-2.875)\psframe(-0.4375,-2.9375)(-0.375,-2.875)\psframe(-0.375,-2.9375)(-0.3125,-2.875)\psframe(-0.3125,-2.9375)(-0.25,-2.875)\psframe(-0.25,-2.9375)(-0.1875,-2.875)\psframe(-0.1875,-2.9375)(-0.125,-2.875)\psframe(-0.125,-2.9375)(-0.0625,-2.875)\psframe(-0.0625,-2.9375)(0,-2.875)\psframe(0,-2.9375)(0.0625,-2.875)\psframe(0.0625,-2.9375)(0.125,-2.875)\psframe(0.125,-2.9375)(0.1875,-2.875)\psframe(0.1875,-2.9375)(0.25,-2.875)\psframe(0.25,-2.9375)(0.3125,-2.875)
\psframe(0.3125,-2.9375)(0.375,-2.875)\psframe(0.375,-2.9375)(0.4375,-2.875)\psframe(0.4375,-2.9375)(0.5,-2.875)\psframe(0.5,-2.9375)(0.5625,-2.875)\psframe(0.5625,-2.9375)(0.625,-2.875)\psframe(0.625,-2.9375)(0.6875,-2.875)\psframe(0.6875,-2.9375)(0.75,-2.875)\psframe(0.75,-2.9375)(0.8125,-2.875)\psframe(0.8125,-2.9375)(0.875,-2.875)\psframe(0.875,-2.9375)(0.9375,-2.875)\psframe(0.9375,-2.9375)(1,-2.875)\psframe(1,-2.9375)(1.0625,-2.875)
\psframe(1.0625,-2.9375)(1.125,-2.875)\psframe(1.125,-2.9375)(1.1875,-2.875)\psframe(1.1875,-2.9375)(1.25,-2.875)\psframe(1.25,-2.9375)(1.3125,-2.875)\psframe(1.3125,-2.9375)(1.375,-2.875)\psframe(1.375,-2.9375)(1.4375,-2.875)\psframe(1.4375,-2.9375)(1.5,-2.875)\psframe(1.5,-2.9375)(1.5625,-2.875)\psframe(1.5625,-2.9375)(1.625,-2.875)\psframe(1.625,-2.9375)(1.6875,-2.875)\psframe(1.6875,-2.9375)(1.75,-2.875)\psframe(1.75,-2.9375)(1.8125,-2.875)\psframe(1.8125,-2.9375)(1.875,-2.875)\psframe(1.875,-2.9375)(1.9375,-2.875)\psframe(1.9375,-2.9375)(2,-2.875)\psframe(2,-2.9375)(2.0625,-2.875)
\psframe(2.0625,-2.9375)(2.125,-2.875)\psframe(2.125,-2.9375)(2.1875,-2.875)\psframe(2.1875,-2.9375)(2.25,-2.875)\psframe(2.25,-2.9375)(2.3125,-2.875)\psframe(2.3125,-2.9375)(2.375,-2.875)\psframe(2.375,-2.9375)(2.4375,-2.875)\psframe(2.4375,-2.9375)(2.5,-2.875)\psframe(2.5,-2.9375)(2.5625,-2.875)\psframe(2.5625,-2.9375)(2.625,-2.875)\psframe(2.625,-2.9375)(2.6875,-2.875)\psframe(2.6875,-2.9375)(2.75,-2.875)\psframe(2.75,-2.9375)(2.8125,-2.875)\psframe(2.8125,-2.9375)(2.875,-2.875)\psframe(2.875,-2.9375)(2.9375,-2.875)\psframe(-2.9375,2.875)(-2.875,2.9375)\psframe(-2.875,2.875)(-2.8125,2.9375)\psframe(-2.8125,2.875)(-2.75,2.9375)\psframe(-2.75,2.875)(-2.6875,2.9375)\psframe(-2.6875,2.875)(-2.625,2.9375)\psframe(-2.625,2.875)(-2.5625,2.9375)\psframe(-2.5625,2.875)(-2.5,2.9375)\psframe(-2.5,2.875)(-2.4375,2.9375)\psframe(-2.4375,2.875)(-2.375,2.9375)\psframe(-2.375,2.875)(-2.3125,2.9375)\psframe(-2.3125,2.875)(-2.25,2.9375)\psframe(-2.25,2.875)(-2.1875,2.9375)\psframe(-2.1875,2.875)(-2.125,2.9375)
\psframe(-2.125,2.875)(-2.0625,2.9375)\psframe(-2.0625,2.875)(-2,2.9375)\psframe(-2,2.875)(-1.9375,2.9375)\psframe(-1.9375,2.875)(-1.875,2.9375)\psframe(-1.875,2.875)(-1.8125,2.9375)\psframe(-1.8125,2.875)(-1.75,2.9375)\psframe(-1.75,2.875)(-1.6875,2.9375)\psframe(-1.6875,2.875)(-1.625,2.9375)\psframe(-1.625,2.875)(-1.5625,2.9375)\psframe(-1.5625,2.875)(-1.5,2.9375)\psframe(-1.5,2.875)(-1.4375,2.9375)\psframe(-1.4375,2.875)(-1.375,2.9375)\psframe(-1.375,2.875)(-1.3125,2.9375)\psframe(-1.3125,2.875)(-1.25,2.9375)\psframe(-1.25,2.875)(-1.1875,2.9375)\psframe(-1.1875,2.875)(-1.125,2.9375)\psframe(-1.125,2.875)(-1.0625,2.9375)\psframe(-1.0625,2.875)(-1,2.9375)\psframe(-1,2.875)(-0.9375,2.9375)\psframe(-0.9375,2.875)(-0.875,2.9375)\psframe(-0.875,2.875)(-0.8125,2.9375)
\psframe(-0.8125,2.875)(-0.75,2.9375)\psframe(-0.75,2.875)(-0.6875,2.9375)\psframe(-0.6875,2.875)(-0.625,2.9375)\psframe(-0.625,2.875)(-0.5625,2.9375)\psframe(-0.5625,2.875)(-0.5,2.9375)\psframe(-0.5,2.875)(-0.4375,2.9375)\psframe(-0.4375,2.875)(-0.375,2.9375)
\psframe(-0.375,2.875)(-0.3125,2.9375)\psframe(-0.3125,2.875)(-0.25,2.9375)\psframe(-0.25,2.875)(-0.1875,2.9375)\psframe(-0.1875,2.875)(-0.125,2.9375)\psframe(-0.125,2.875)(-0.0625,2.9375)\psframe(-0.0625,2.875)(0,2.9375)\psframe(0,2.875)(0.0625,2.9375)\psframe(0.0625,2.875)(0.125,2.9375)\psframe(0.125,2.875)(0.1875,2.9375)\psframe(0.1875,2.875)(0.25,2.9375)\psframe(0.25,2.875)(0.3125,2.9375)\psframe(0.3125,2.875)(0.375,2.9375)\psframe(0.375,2.875)(0.4375,2.9375)\psframe(0.4375,2.875)(0.5,2.9375)\psframe(0.5,2.875)(0.5625,2.9375)\psframe(0.5625,2.875)(0.625,2.9375)\psframe(0.625,2.875)(0.6875,2.9375)\psframe(0.6875,2.875)(0.75,2.9375)\psframe(0.75,2.875)(0.8125,2.9375)\psframe(0.8125,2.875)(0.875,2.9375)\psframe(0.875,2.875)(0.9375,2.9375)\psframe(0.9375,2.875)(1,2.9375)\psframe(1,2.875)(1.0625,2.9375)
\psframe(1.0625,2.875)(1.125,2.9375)\psframe(1.125,2.875)(1.1875,2.9375)\psframe(1.1875,2.875)(1.25,2.9375)\psframe(1.25,2.875)(1.3125,2.9375)\psframe(1.3125,2.875)(1.375,2.9375)
\psframe(1.375,2.875)(1.4375,2.9375)\psframe(1.4375,2.875)(1.5,2.9375)\psframe(1.5,2.875)(1.5625,2.9375)\psframe(1.5625,2.875)(1.625,2.9375)\psframe(1.625,2.875)(1.6875,2.9375)\psframe(1.6875,2.875)(1.75,2.9375)\psframe(1.75,2.875)(1.8125,2.9375)\psframe(1.8125,2.875)(1.875,2.9375)\psframe(1.875,2.875)(1.9375,2.9375)\psframe(1.9375,2.875)(2,2.9375)\psframe(2,2.875)(2.0625,2.9375)\psframe(2.0625,2.875)(2.125,2.9375)\psframe(2.125,2.875)(2.1875,2.9375)\psframe(2.1875,2.875)(2.25,2.9375)\psframe(2.25,2.875)(2.3125,2.9375)\psframe(2.3125,2.875)(2.375,2.9375)\psframe(2.375,2.875)(2.4375,2.9375)\psframe(2.4375,2.875)(2.5,2.9375)\psframe(2.5,2.875)(2.5625,2.9375)\psframe(2.5625,2.875)(2.625,2.9375)\psframe(2.625,2.875)(2.6875,2.9375)\psframe(2.6875,2.875)(2.75,2.9375)\psframe(2.75,2.875)(2.8125,2.9375)\psframe(2.8125,2.875)(2.875,2.9375)\psframe(2.875,2.875)(2.9375,2.9375)\psframe(-2.9375,-2.9375)(-2.875,-2.875)\psframe(-2.9375,-2.875)(-2.875,-2.8125)
\psframe(-2.9375,-2.8125)(-2.875,-2.75)\psframe(-2.9375,-2.75)(-2.875,-2.6875)
\psframe(-2.9375,-2.6875)(-2.875,-2.625)\psframe(-2.9375,-2.625)(-2.875,-2.5625)\psframe(-2.9375,-2.5625)(-2.875,-2.5)\psframe(-2.9375,-2.5)(-2.875,-2.4375)\psframe(-2.9375,-2.4375)(-2.875,-2.375)\psframe(-2.9375,-2.375)(-2.875,-2.3125)\psframe(-2.9375,-2.3125)(-2.875,-2.25)\psframe(-2.9375,-2.25)(-2.875,-2.1875)\psframe(-2.9375,-2.1875)(-2.875,-2.125)\psframe(-2.9375,-2.125)(-2.875,-2.0625)\psframe(-2.9375,-2.0625)(-2.875,-2)\psframe(-2.9375,-2)(-2.875,-1.9375)\psframe(-2.9375,-1.9375)(-2.875,-1.875)\psframe(-2.9375,-1.875)(-2.875,-1.8125)\psframe(-2.9375,-1.8125)(-2.875,-1.75)\psframe(-2.9375,-1.75)(-2.875,-1.6875)\psframe(-2.9375,-1.6875)(-2.875,-1.625)\psframe(-2.9375,-1.625)(-2.875,-1.5625)\psframe(-2.9375,-1.5625)(-2.875,-1.5)\psframe(-2.9375,-1.5)(-2.875,-1.4375)\psframe(-2.9375,-1.4375)(-2.875,-1.375)\psframe(-2.9375,-1.375)(-2.875,-1.3125)\psframe(-2.9375,-1.3125)(-2.875,-1.25)\psframe(-2.9375,-1.25)(-2.875,-1.1875)
\psframe(-2.9375,-1.1875)(-2.875,-1.125)\psframe(-2.9375,-1.125)(-2.875,-1.0625)\psframe(-2.9375,-1.0625)(-2.875,-1)
\psframe(-2.9375,-1)(-2.875,-0.9375)\psframe(-2.9375,-0.9375)(-2.875,-0.875)\psframe(-2.9375,-0.875)(-2.875,-0.8125)\psframe(-2.9375,-0.8125)(-2.875,-0.75)\psframe(-2.9375,-0.75)(-2.875,-0.6875)\psframe(-2.9375,-0.6875)(-2.875,-0.625)\psframe(-2.9375,-0.625)(-2.875,-0.5625)\psframe(-2.9375,-0.5625)(-2.875,-0.5)\psframe(-2.9375,-0.5)(-2.875,-0.4375)\psframe(-2.9375,-0.4375)(-2.875,-0.375)\psframe(-2.9375,-0.375)(-2.875,-0.3125)\psframe(-2.9375,-0.3125)(-2.875,-0.25)\psframe(-2.9375,-0.25)(-2.875,-0.1875)\psframe(-2.9375,-0.1875)(-2.875,-0.125)\psframe(-2.9375,-0.125)(-2.875,-0.0625)\psframe(-2.9375,-0.0625)(-2.875,0)\psframe(-2.9375,0)(-2.875,0.0625)\psframe(-2.9375,0.0625)(-2.875,0.125)\psframe(-2.9375,0.125)(-2.875,0.1875)\psframe(-2.9375,0.1875)(-2.875,0.25)\psframe(-2.9375,0.25)(-2.875,0.3125)\psframe(-2.9375,0.3125)(-2.875,0.375)\psframe(-2.9375,0.375)(-2.875,0.4375)\psframe(-2.9375,0.4375)(-2.875,0.5)\psframe(-2.9375,0.5)(-2.875,0.5625)\psframe(-2.9375,0.5625)(-2.875,0.625)
\psframe(-2.9375,0.625)(-2.875,0.6875)\psframe(-2.9375,0.6875)(-2.875,0.75)\psframe(-2.9375,0.75)(-2.875,0.8125)\psframe(-2.9375,0.8125)(-2.875,0.875)\psframe(-2.9375,0.875)(-2.875,0.9375)\psframe(-2.9375,0.9375)(-2.875,1)\psframe(-2.9375,1)(-2.875,1.0625)\psframe(-2.9375,1.0625)(-2.875,1.125)\psframe(-2.9375,1.125)(-2.875,1.1875)\psframe(-2.9375,1.1875)(-2.875,1.25)\psframe(-2.9375,1.25)(-2.875,1.3125)\psframe(-2.9375,1.3125)(-2.875,1.375)\psframe(-2.9375,1.375)(-2.875,1.4375)\psframe(-2.9375,1.4375)(-2.875,1.5)\psframe(-2.9375,1.5)(-2.875,1.5625)\psframe(-2.9375,1.5625)(-2.875,1.625)\psframe(-2.9375,1.625)(-2.875,1.6875)\psframe(-2.9375,1.6875)(-2.875,1.75)\psframe(-2.9375,1.75)(-2.875,1.8125)\psframe(-2.9375,1.8125)(-2.875,1.875)\psframe(-2.9375,1.875)(-2.875,1.9375)\psframe(-2.9375,1.9375)(-2.875,2)\psframe(-2.9375,2)(-2.875,2.0625)\psframe(-2.9375,2.0625)(-2.875,2.125)\psframe(-2.9375,2.125)(-2.875,2.1875)\psframe(-2.9375,2.1875)(-2.875,2.25)\psframe(-2.9375,2.25)(-2.875,2.3125)
\psframe(-2.9375,2.3125)(-2.875,2.375)
\psframe(-2.9375,2.375)(-2.875,2.4375)\psframe(-2.9375,2.4375)(-2.875,2.5)\psframe(-2.9375,2.5)(-2.875,2.5625)\psframe(-2.9375,2.5625)(-2.875,2.625)\psframe(-2.9375,2.625)(-2.875,2.6875)\psframe(-2.9375,2.6875)(-2.875,2.75)\psframe(-2.9375,2.75)(-2.875,2.8125)\psframe(-2.9375,2.8125)(-2.875,2.875)\psframe(-2.9375,2.875)(-2.875,2.9375)\psframe(2.875,-2.9375)(2.9375,-2.875)\psframe(2.875,-2.875)(2.9375,-2.8125)\psframe(2.875,-2.8125)(2.9375,-2.75)\psframe(2.875,-2.75)(2.9375,-2.6875)\psframe(2.875,-2.6875)(2.9375,-2.625)\psframe(2.875,-2.625)(2.9375,-2.5625)\psframe(2.875,-2.5625)(2.9375,-2.5)\psframe(2.875,-2.5)(2.9375,-2.4375)\psframe(2.875,-2.4375)(2.9375,-2.375)\psframe(2.875,-2.375)(2.9375,-2.3125)\psframe(2.875,-2.3125)(2.9375,-2.25)\psframe(2.875,-2.25)(2.9375,-2.1875)\psframe(2.875,-2.1875)(2.9375,-2.125)\psframe(2.875,-2.125)(2.9375,-2.0625)\psframe(2.875,-2.0625)(2.9375,-2)\psframe(2.875,-2)(2.9375,-1.9375)\psframe(2.875,-1.9375)(2.9375,-1.875)\psframe(2.875,-1.875)(2.9375,-1.8125)
\psframe(2.875,-1.8125)(2.9375,-1.75)\psframe(2.875,-1.75)(2.9375,-1.6875)\psframe(2.875,-1.6875)(2.9375,-1.625)\psframe(2.875,-1.625)(2.9375,-1.5625)\psframe(2.875,-1.5625)(2.9375,-1.5)\psframe(2.875,-1.5)(2.9375,-1.4375)\psframe(2.875,-1.4375)(2.9375,-1.375)\psframe(2.875,-1.375)(2.9375,-1.3125)\psframe(2.875,-1.3125)(2.9375,-1.25)\psframe(2.875,-1.25)(2.9375,-1.1875)\psframe(2.875,-1.1875)(2.9375,-1.125)\psframe(2.875,-1.125)(2.9375,-1.0625)\psframe(2.875,-1.0625)(2.9375,-1)\psframe(2.875,-1)(2.9375,-0.9375)\psframe(2.875,-0.9375)(2.9375,-0.875)\psframe(2.875,-0.875)(2.9375,-0.8125)\psframe(2.875,-0.8125)(2.9375,-0.75)\psframe(2.875,-0.75)(2.9375,-0.6875)\psframe(2.875,-0.6875)(2.9375,-0.625)\psframe(2.875,-0.625)(2.9375,-0.5625)\psframe(2.875,-0.5625)(2.9375,-0.5)\psframe(2.875,-0.5)(2.9375,-0.4375)\psframe(2.875,-0.4375)(2.9375,-0.375)\psframe(2.875,-0.375)(2.9375,-0.3125)\psframe(2.875,-0.3125)(2.9375,-0.25)\psframe(2.875,-0.25)(2.9375,-0.1875)\psframe(2.875,-0.1875)(2.9375,-0.125)\psframe(2.875,-0.125)
(2.9375,-0.0625)
\psframe(2.875,-0.0625)(2.9375,0)\psframe(2.875,0)(2.9375,0.0625)\psframe(2.875,0.0625)(2.9375,0.125)\psframe(2.875,0.125)(2.9375,0.1875)\psframe(2.875,0.1875)(2.9375,0.25)\psframe(2.875,0.25)(2.9375,0.3125)\psframe(2.875,0.3125)(2.9375,0.375)\psframe(2.875,0.375)(2.9375,0.4375)\psframe(2.875,0.4375)(2.9375,0.5)\psframe(2.875,0.5)(2.9375,0.5625)\psframe(2.875,0.5625)(2.9375,0.625)\psframe(2.875,0.625)(2.9375,0.6875)\psframe(2.875,0.6875)(2.9375,0.75)\psframe(2.875,0.75)(2.9375,0.8125)\psframe(2.875,0.8125)(2.9375,0.875)\psframe(2.875,0.875)(2.9375,0.9375)\psframe(2.875,0.9375)(2.9375,1)\psframe(2.875,1)(2.9375,1.0625)\psframe(2.875,1.0625)(2.9375,1.125)\psframe(2.875,1.125)(2.9375,1.1875)\psframe(2.875,1.1875)(2.9375,1.25)\psframe(2.875,1.25)(2.9375,1.3125)\psframe(2.875,1.3125)(2.9375,1.375)\psframe(2.875,1.375)(2.9375,1.4375)\psframe(2.875,1.4375)(2.9375,1.5)\psframe(2.875,1.5)(2.9375,1.5625)\psframe(2.875,1.5625)(2.9375,1.625)\psframe(2.875,1.625)(2.9375,1.6875)\psframe(2.875,1.6875)(2.9375,1.75)
\psframe(2.875,1.75)(2.9375,1.8125)\psframe(2.875,1.8125)(2.9375,1.875)\psframe(2.875,1.875)(2.9375,1.9375)\psframe(2.875,1.9375)(2.9375,2)\psframe(2.875,2)(2.9375,2.0625)\psframe(2.875,2.0625)(2.9375,2.125)\psframe(2.875,2.125)(2.9375,2.1875)\psframe(2.875,2.1875)(2.9375,2.25)\psframe(2.875,2.25)(2.9375,2.3125)\psframe(2.875,2.3125)(2.9375,2.375)\psframe(2.875,2.375)(2.9375,2.4375)\psframe(2.875,2.4375)(2.9375,2.5)\psframe(2.875,2.5)(2.9375,2.5625)\psframe(2.875,2.5625)(2.9375,2.625)\psframe(2.875,2.625)(2.9375,2.6875)\psframe(2.875,2.6875)(2.9375,2.75)\psframe(2.875,2.75)(2.9375,2.8125)\psframe(2.875,2.8125)(2.9375,2.875)\psframe(2.875,2.875)(2.9375,2.9375)

	\psframe[linestyle=dotted, linewidth=1pt](-3,-3)(3,3)
\end{pspicture}
\end{center}
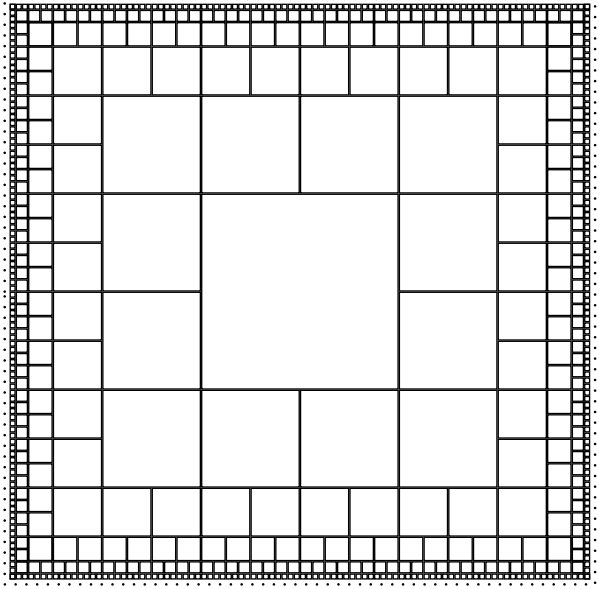
\centerline{{\small \textbf{Figure 2:} Whitney Decomposition of $D$ when $n=2$} }

\medskip\noindent \textbf{Step 5: Construction of $w$} 

\smallskip\noindent We now define our map $w=w_{n}$ as follows. We shall define it on $D$ instead of $Q^{(n)}$ (then we may just take $\bar{w}(x) = \frac{1}{3} w(3x)$).

   Let $Q$ be a cube in $\mathscr{Q}_{k}$ for $k\geq 2$. Write $Q=x+(-r,r)^{n}$ where $x$ is the centre of the cube, and $2r$ is the side length. Then note that in all but $2$ opposite faces, a face of $Q$ is shared with other cubes in $\mathscr{Q}_{k}$. One of the remaining faces is shared with (part of) a face of a larger cube with side length $4r$ from $\mathscr{Q}_{k-1}$, and the other opposite face is shared with the faces of $2^{n-1}$ smaller cubes of length $r$ in $\mathscr{Q}_{k+1}$. We apply the map $v$ from step 4, appropriately scaled, to map points in $Q$ to a one-dimensional frame of $Q$, where the modification of $u$ occurs on the cone corresponding to this latter face, which is split into a finer subframe. That is, for $y\in Q$,
\begin{equation*}
 w(y):= x+ r\,v\Big(\frac{y-x}{r}\Big)\,\textrm{,}
\end{equation*}
where $v=u$ on all cones of $(-1,1)^{n}$ where 
\begin{equation*}
x+r\frac{y-x}{\|y-x\|}
\end{equation*}
lies on a face of $Q$ that is shared with a cube of the same or larger size. Otherwise, if this expression lies on the one face that is shared with $2^{n-1}$ faces of cubes in $\mathscr{Q}_{k+1}$, we use the definition of $v$ as given in \eqref{mapvnminus} and \eqref{mapv}. Now note that
\begin{align*}
 \int_{Q}|\nabla w (y)|^{p}\ud y &= \int_{Q}\Big|\nabla v\Big(\frac{y-x}{r}\Big)\Big|^{p}\ud y \\
&= r^{n} \int_{(-1,1)^{n}}|\nabla v(y))|^{p}\ud y\\
&= 2^{-n}\Ln(Q) \int_{(-1,1)^{n}}|\nabla v(y))|^{p}\ud y\,\textrm{.}
\end{align*}
We do this for all cubes $Q\in \mathscr{Q}_{k}$ for all $k\geq 2$. Moreover, note that by our construction there are no discontinuities on the boundaries $\partial Q$ of these cubes. Hence we have
\begin{align*}
 \int_{D\setminus [-1,1]^{n}}|\nabla w|^{p}\ud x &= \sum_{k=2}^{\infty}\sum_{Q\in\mathscr{Q}_{k}} \int_{Q}|\nabla w|^{p}\ud x \\
&= \sum_{k=2}^{\infty}\sum_{Q\in\mathscr{Q}_{k}} 2^{-n}\Ln(Q) \int_{(-1,1)^{n}}|\nabla v|^{p}\ud x\\
&= 2^{-n}\Ln(D\setminus [-1,1]^{n})\int_{(-1,1)^{n}}|\nabla v|^{p}\ud x\\
&= (3^{n}-1)\int_{(-1,1)^{n}}|\nabla v|^{p}\ud x\,\textrm{,}
\end{align*}
and so $w\in W^{1,p}(D\setminus (-1,1)^{n};\Rn)$. 

Note that on the central cube $Q_{1}$, \textit{every} face is shared with the $2^{n-1}$ faces of cubes in $\mathscr{Q}_{2}$, not just one. Hence in this case we modify the definition of $u$ as in Step 3, not just on one cone, but all cones, and let $w$ be equal to such a map on $Q_{1}$. Again, we observe that $w$ has no discontinuities on $\partial Q_{1}$. 

Therefore we have $w\in W^{1,p}(D;\Rn)$, and $\textrm{rank}(\nabla w (x)) = 1$ for $\Ln$-almost all $x\in D$ (in fact, even $\mathscr{H}^{n-1}$ almost all $x$). Now we shall show that $w=\iota$ on $\partial D$. First note that if $\textrm{dist}(x,\partial D)<\epsilon<2$, then $x\in \overline{Q}$ for some $Q\in\mathscr{Q}_{k}$, where $k\geq k_{0}$ and $2^{3-k_{0}}<\epsilon$. Since $w(x)\in \overline{Q}$, we have 
\begin{equation*}
 |w(x)-x|\leq\textrm{diam}(Q) < \epsilon\,\textrm{,}
\end{equation*}
Let $(R_{h})\subset (0,1)$ be an increasing sequence with $R_{h}\nearrow 1$, and let $\rho_{h}\in C^{1}_{c}(D)$ be a cut-off function such that $0\leq \rho\leq 1$, $\rho_{h}=1$ on $R_{h}D:=(-3R_{h},3R_{h})^{n}$, and 
\begin{equation*}
 |\nabla\rho_{h}|\leq \frac{c}{1-R_{h}}
\end{equation*}
for some fixed constant $c>0$ independent of $h$. Consider $w_{h}:= \rho_{h}w+(1-\rho_{h})\iota$. Then note $(w_{h})\subset W^{1,p}_{\iota}(D;\Rn)$, and
\begin{align*}
 \|w-w_{h}\|_{\infty} &= \|(1-\rho_{h})(w-\iota)\|_{\infty}\\
&\leq \sup_{x\in D\setminus R_{h}D}|w(x)-x|\\
&\longrightarrow 0\quad\textrm{as }h\rightarrow\infty\,\textrm{,}
\end{align*}
so $w_{h}\rightarrow w$ in $L^{\infty}(D;\Rn)$. Moreover, we have
\begin{equation*}
 \nabla w_{h} = \rho_{h}\nabla w + (1-\rho_{h})I + (w-\iota)\otimes \nabla\rho_{h}\,\textrm{,}
\end{equation*}
and (for positive constants $c$, independent of $h$, that may not be the same from line to line), 
\begin{align*}
 \int_{D}|\nabla w-\nabla w_{h}|^{p}\ud x &=\int_{D\setminus R_{h}D} |(1-\rho_{h})(\nabla w - I)+(w-\iota)\otimes \nabla\rho_{h}|^{p}\ud x \\
&\leq c\int_{D\setminus R_{h}D}|\nabla w|^{p}+1\ud x + \frac{c}{1-R_{h}}\int_{D\setminus R_{h}D} |w-\iota|^{p}\ud x\\
&\leq c\int_{D\setminus R_{h}D}|\nabla w|^{p}+1\ud x +c \Bigg(\frac{1-R_{h}^{n}}{1-R_{h}}\Bigg)\sup_{x\in D\setminus R_{h}D}|w(x)-x|\\
&\longrightarrow 0\quad\textrm{as }h\rightarrow\infty\,\textrm{.}
\end{align*}
Hence $w_{h}\rightarrow w$ strongly in $W^{1,p}(D;\Rn)$, so we have $w\in W^{1,p}_{\iota}(D;\Rn)$ as required. 

We have shown that if $w_{n-1}$ satisfies all required properties of Lemma \ref{lemma2}, then so does $w=w_{n}$. Thus, by Step 1 and induction, the lemma is proved. 

\nopagebreak\hfill $\Box$

\bigskip\noindent \textbf{Step 6: Proof of Lemma \ref{lemma3}} 

\smallskip\noindent The proof of this lemma follows almost exactly the same steps as Lemma \ref{lemma2}. The only difference is that the base case is $n=k$, and we start by defining the map  $\tilde{w}_{k}(x):=x/\|x\|$ on $Q^{(k)}$. By an entirely similar arguement as in Step 1 (see also \cite{BallMurat}), we establish that $\tilde{w}_{k}$ satisfies the required properties of the lemma when $n=k$. For larger $n$, we then apply the inductive hypothesis that for all $k\leq d < n$, we have constructed a map $\tilde{w}_{d}$ satisfying the lemma in dimension $d$. We define $\tilde{u}$ on $Q^{(n)}$ precisely as in Step 2, using $\tilde{w}_{n-1}$ instead of $w_{n-1}$ and modify it on cones as in Step 3. The Whitney Decomposition of $Q^{(n)}$ in Step 4 remains the same, and in Step 5 we define $\tilde{w}=\tilde w_{n}$ using this decomposition and the modified map $\tilde{u}$. 

\nopagebreak\hfill $\Box$

\bigskip\noindent \textbf{Remark on generalising the above construction} 

 \noindent Recall that a general Whitney Decomposition allows us to partition a general open, bounded subset $\Omega\subset\Rn$ into closed diadic cubes $(Q_{j})_{j\in\mathbb{N}}$ with pairwise disjoint interior, satisfying
 \begin{equation*}
  \mathrm{diam}(Q_{j})\leq \mathrm{dist}(Q_{j};\partial\Omega)\leq 4 \mathrm{diam}(Q_{j})\quad\textrm{for all }j\,\textrm{.}
 \end{equation*}
 In this connection, we refer to, for example \cite{whitney, harmonic}. Hence we can refine the statements of Lemma \ref{lemma2} and Lemma \ref{lemma3} respectively so that they satisfy the requisite properties on such general $\Omega$. For example, we have the following corollary. It can either be proved by an easy modification of the above proof, taking care of the cones in each individual cube where the map $u$ from Step 2 needs to be modified, or by simply using the results in these lemmas. We proceed using the latter method.

\begin{cor}
 Let $\Omega\subset\Rn$ be open and bounded. Suppose $g\colon\Omega\rightarrow\RN$ is an affine map. Then for $1\leq p < 2$ there exists a map $u\colon\Omega\rightarrow\RN$ such that $u\in W^{1,p}_{g}(\Omega;\RN)$ and $\textrm{rank}(\nabla u(x))\leq 1$ for almost all $x\in\Omega$. 

More generally, if $2\leq k\leq n$ and $1\leq p <k$, there exists a map $u\colon\Omega\rightarrow\RN$ such that $u\in W^{1,p}_{g}(\Omega;\RN)$ and $\textrm{rank}(\nabla u(x))\leq k-1$ for almost all $x\in\Omega$. 

\end{cor}

\begin{proof}
 Take a general Whitney Decomposition of $\Omega$ as described above, where the cubes have sides parallel to the coordinate axes. First define $v\colon\Omega\rightarrow\Rn$ as follows: for a given cube $Q_{j}=x+r(-1,1)^{n}$ in this decomposition, let
\begin{equation}
 v(y)= x+ r w\Big( \frac{y-x}{r}\Big)\,\textrm{,} \label{genw}
\end{equation}
where $w$ is the map from Lemma \ref{lemma2}. Do this for every cube in the decomposition. Writing $g(y)= z+\xi y$ for some $z\in\R^{N}$, $\xi\in\RNn$, let $u(y)= z+\xi v(y)$ . Then it is straightforward to verify that $u$ satisfies the required properties in (first statement of) the corollary (see the proof of Theorem \ref{thm1} above). For the general statement, use $\tilde{w}$ from Lemma \ref{lemma3} instead of $w$ in \eqref{genw}.
\end{proof}

\providecommand{\bysame}{\leavevmode\hbox to3em{\hrulefill}\thinspace}
\providecommand{\MR}{\relax\ifhmode\unskip\space\fi MR }
\providecommand{\MRhref}[2]{%
  \href{http://www.ams.org/mathscinet-getitem?mr=#1}{#2}
}
\providecommand{\href}[2]{#2}

\end{document}